\documentclass[11pt,a4paper]{article}
\usepackage[T1]{fontenc}
\usepackage[utf8]{inputenc}
\usepackage{amsmath,amssymb,amsthm,mathtools}
\IfFileExists{rsfso.sty}{\usepackage{rsfso}}{\usepackage{mathrsfs}}
\IfFileExists{newtxtext.sty}{\usepackage{newtxtext,newtxmath}}{\usepackage{lmodern}}
\usepackage{microtype}
\usepackage{geometry}
\usepackage{enumitem}
\usepackage[hidelinks]{hyperref}
\hypersetup{
  pdftitle={Global Strong Solutions for the Ternary Maxwell--Stefan Diffusion System},
  pdfauthor={Dieter Bothe}
}
\setlist[itemize]{topsep=3pt,itemsep=1pt,parsep=0pt}
\setlist[enumerate]{topsep=3pt,itemsep=1pt,parsep=0pt}
\allowdisplaybreaks

\newtheorem{theorem}{Theorem}[section]
\newtheorem{proposition}[theorem]{Proposition}
\newtheorem{lemma}[theorem]{Lemma}

\theoremstyle{definition}

\theoremstyle{remark}
\newtheorem{remark}[theorem]{Remark}

\newcommand{\R}{\mathbb R}
\newcommand{\sym}{\operatorname{sym}}
\newcommand{\Div}{\operatorname{div}}
\newcommand{\esssup}{\operatorname*{ess\,sup}}
\newcommand{\essinf}{\operatorname*{ess\,inf}}
\newcommand{\dd}{\,\mathrm d}

\title{Global Strong Solutions for the Ternary\\Maxwell--Stefan Diffusion System}
\author{Dieter Bothe\\
\small Mathematical Modeling and Analysis, Department of Mathematics,\\
\small Technische Universit\"at Darmstadt, Peter-Gr\"unberg-Stra\ss e 10,\\
\small 64287 Darmstadt, Germany\\
\small \texttt{bothe@mma.tu-darmstadt.de}}
\date{}

\begin{document}
\maketitle

\begin{abstract}
We prove global strong well-posedness for the normalized ideal, isothermal, isobaric ternary
Maxwell--Stefan diffusion system with equal molar masses, unit total density, and no reactions.
On bounded $C^{2+\alpha_0}$ domains, arbitrary positive constant binary frictions and compatible
nonnegative initial data are allowed, without a smallness condition.
For distinct frictions, trace--determinant variables produce two separated spectral roots satisfying
scalar parabolic equations with antisymmetric quadratic coupling.  Euler--Poisson--Darboux
one-sided truncations, stabilized by the mixing entropy, yield boundary De Giorgi oscillation decay
and positive-time H\"older regularity.  The mixing-entropy Hessian then provides a local
symmetrizer; conormal Schauder estimates give uniform $C^{1+a}$ control, and frozen maximal
$L_p$-regularity yields a uniform terminal trace bound excluding finite-time breakdown.  Repeated
frictions reduce to triangular scalar parabolic systems.  Every component with positive conserved
mass is strictly positive for positive time.
\end{abstract}

\medskip
\noindent\textbf{Keywords.}
Maxwell--Stefan diffusion; ternary cross-diffusion; global strong solutions; De Giorgi regularity;
entropy symmetrization; conormal Schauder estimates.

\smallskip
\noindent\textbf{2020 Mathematics Subject Classification.}
35K51, 35K59 (primary); 35B65, 35D35, 76R50 (secondary).

\section{Introduction and main result}\label{sec:intro}

Maxwell--Stefan diffusion couples the species fluxes through pairwise friction laws and leads,
after elimination of the flux constraint, to a generally nonsymmetric cross-diffusion system.
For the isothermal, isobaric diffusion problem, normal ellipticity yields local strong
well-posedness; see Bothe~\cite{Bothe2011} and the maximal-$L_p$ theory of
Herberg--Meyries--Pr\"uss--Wilke~\cite{HerbergMeyriesPruessWilke2017}, which also gives
positivity and instantaneous regularization on bounded domains.  The same functional-analytic
approach has also been applied to multicomponent systems coupled to fluid flow, reactions, mass
transfer, and phase transition.  In
Bothe--Pr\"uss~\cite{BothePruess2017}, Maxwell--Stefan diffusion is incorporated into a
reactive two-phase model with mass transfer and phase transition, for which an $L_p$
well-posedness and state-manifold theory is developed.  Bothe--Druet established local strong
well-posedness and global existence near equilibrium for broad compressible and incompressible
multicomponent flow models~\cite{BotheDruet2021,BotheDruetIncompressible2021}.  A global
perturbative Cauchy theory for a flux-incompressible Maxwell--Stefan system was established by
Bondesan--Briant~\cite{BondesanBriant2022}.

Related $L_p$-based strong-solution and continuation methods for chemically coupled parabolic
transport systems have been developed for electrokinetic Navier--Stokes--Nernst--Planck--Poisson
flows~\cite{BotheFischerSaal2014}, for reaction--advection--anisotropic-diffusion systems using
optimal maximal regularity for time-dependent diffusion operators~\cite{BotheFischerPierreRolland2017},
and for bulk--surface reaction--diffusion and heterogeneous-catalysis systems
~\cite{BotheKohneMaierSaal2017,AugnerBothe2021,AugnerBothe2024}.  Although these problems have
different diffusion laws and couplings, they share with the present argument the use of strong
solution spaces, positivity or physical-state constraints, global a priori estimates, and
continuation arguments based on parabolic regularity.

At the weak-solution level, several global results are available.  J\"ungel--Stelzer
proved global bounded weak solutions and exponential convergence for general multicomponent
Maxwell--Stefan systems~\cite{JungelStelzer2013}; related reactive systems were treated by
Daus--J\"ungel--Tang~\cite{DausJungelTang2020}.  Marion--Temam obtained global existence for
fully nonlinear multispecies reaction--diffusion systems whose diffusion coefficients arise from
the Stefan--Maxwell relations~\cite{MarionTemam2015}.  Cross-diffusion also appears as the
fast-reaction limit of reaction--diffusion systems; Bothe--Pierre--Rolland~\cite{BothePierreRolland2012}
combined entropy estimates with the limiting argument and obtained local strong solvability for
the resulting cross-diffusion problem.  Relative-entropy methods yield
weak--strong uniqueness in the bounded-domain Maxwell--Stefan setting
\cite{HuoJungelTzavaras2022}; subsequent work includes uniqueness of renormalized solutions,
with an application to weak solutions with bounded fluxes~\cite{GeorgiadisKimTzavaras2024},
and a weak--strong uniqueness result for broad cross-diffusion classes that include
Maxwell--Stefan systems, with relaxed regularity and positivity assumptions on the strong
solution~\cite{GeltnerHeitzingerJungel2026}.  Weak solutions also satisfy the
entropy equality without anomalous dissipation~\cite{BerselliGeorgiadisTzavaras2025}.
Braukhoff--Raithel--Zamponi proved partial $C^{1,\alpha}$ regularity for weak
Maxwell--Stefan solutions within an entropy-structured cross-diffusion
framework~\cite{BraukhoffRaithelZamponi2022}.  Georgiadis~\cite{Georgiadis2026} also obtains
global weak solutions for the isothermal Maxwell--Stefan system on the whole space $\mathbb R^3$.

Global strong solvability for arbitrary data is known in special coefficient configurations.  In
the ternary case Boudin--Grec--Salvarani~\cite{BoudinGrecSalvarani2012} exploited a triangular
structure arising from two equal binary coefficients to obtain unique global smooth solutions.
To the best of our knowledge, the literature cited above does not contain global strong
well-posedness for arbitrary compatible data in the bounded-domain ternary Maxwell--Stefan
diffusion problem with three unrestricted positive binary frictions: the available general results
are instead local strong, perturbative or near-equilibrium strong, or global weak results.  We prove
global strong well-posedness in this setting for arbitrary compatible nonnegative strong data,
including data on the boundary of the concentration simplex.

A methodological comparison is provided by the author's preprint~\cite{BotheAdditive2026},
which treats Maxwell--Stefan systems with additive frictions by means of Stieltjes root
coordinates, Euler--Poisson--Darboux (EPD) identities, entropy-stabilized truncations, De Giorgi
regularity, and a terminal continuation argument.  The present ternary problem is not restricted to the
positive additive class: although every symmetric ternary friction triple has a unique real
pair-sum representation, its generators need not be positive.  Nevertheless the two-root
Stieltjes/EPD structure persists.  Section~\ref{subsec:additive-bridge} identifies the precise
affine relation between the two coordinate systems.  All estimates needed for the unrestricted
ternary theorem are stated and proved below; the preprint is used for methodological comparison
and attribution rather than as a substitute for the arguments required here.

A second interpretation follows from the core-diagonal thermodynamic closure.
Bothe and Druet~\cite[Sections~8--9]{BotheDruet2023} show that this closure generates the
multicomponent Darken relation and that every ternary Maxwell--Stefan system is core-diagonal.  In the present constant-friction
setting the corresponding three tracer diffusivities are a common scalar multiple of one fixed
species vector.  Thus the tracer diffusivity vector in the ternary problem considered here lies on a fixed ray
within the Darken class.  Section~\ref{subsec:darken-bridge} gives the exact formula and shows that the quadratic
trace--determinant polynomial is precisely the Stieltjes numerator for the fixed pole vector
obtained by removing the common scalar factor from the diagonal tracer mobility.  Thus the two root coordinates used below coincide with the $N=3$ Stieltjes coordinates of the
core-diagonal closure.

The theorem proved here does not impose the positive-generator restriction of the additive class.
No constitutive relation among the three positive physical frictions is imposed: the unique real
pair-sum generator may be nonpositive, equivalently one of the strict triangle inequalities for
the friction triple may fail.  Moreover the initial concentrations may lie on the boundary of the
simplex.  The unrestricted ternary reduction still yields two scalar root equations with an
antisymmetric quadratic coupling; exact one-sided EPD truncations and the mixing-entropy
stabilization yield the estimates used in the boundary De Giorgi argument and the terminal
continuation step.
The result below therefore applies to arbitrary positive ternary frictions and arbitrary compatible
nonnegative strong data.  No smallness condition is imposed, and the result holds in every finite
spatial dimension.

Let
\[
\mathbf e=(1,1,1)^{\sf T},\qquad \mathbb E=\{z\in\R^3:z\cdot\mathbf e=0\}.
\]
For $p>d+2$ set
\[
X_{\gamma,p}
=\{z\in W_p^{2-2/p}(\Omega;\mathbb E):\partial_\nu z=0\}.
\]
A global strong solution means that for every $T>0$,
\[
y-\mathbf e/3\in
W_p^1(0,T;L_p(\Omega;\mathbb E))
\cap L_p(0,T;W_p^2(\Omega;\mathbb E))
\cap C([0,T];X_{\gamma,p}),
\]
and that the Maxwell--Stefan equations and the no-flux condition hold almost everywhere.
Here ``classical for every positive time'' means that for every $0<\tau<T<\infty$ the solution
has
\[
y\in C^1([\tau,T];C(\overline\Omega;\mathbb R^3))
\cap C([\tau,T];C^2(\overline\Omega;\mathbb R^3)),
\]
with the equations and the boundary condition satisfied pointwise; the local theory recalled below
also yields the corresponding positive-time H\"older refinements.

\begin{theorem}[Global strong solvability]\label{thm:main}
Consider the normalized ideal isothermal, isobaric ternary Maxwell--Stefan diffusion model of
Section~\ref{sec:model}, with equal molar masses, unit total density, and zero reaction.
Let $d\ge1$, let $0<\alpha_0<1$, and let $\Omega\subset\R^d$ be bounded and connected with
$C^{2+\alpha_0}$ boundary.  Let $p>d+2$, and let $f_{ij}=f_{ji}>0$ be arbitrary constant binary frictions.  Suppose
\[
\begin{gathered}
y^0\in W_p^{2-2/p}(\Omega;\R^3),\qquad
 y_i^0\ge0,\qquad \sum_{i=1}^3y_i^0=1\quad\text{on }\overline\Omega,\\
\partial_\nu y^0=0\quad\text{on }\partial\Omega\text{ in the trace sense}.
\end{gathered}
\]
Then this problem, with homogeneous no-flux boundary conditions, has a unique global strong solution.  The solution remains in the closed concentration simplex and is
classical for every positive time.  If $\int_\Omega y_i^0\,\dd x>0$, then
$y_i(t,x)>0$ for every $t>0$ and $x\in\overline\Omega$.
\end{theorem}

\paragraph{Outline of the proof.}
For pairwise distinct frictions the contradiction argument has four stages:
\begin{enumerate}[label=\textup{(\roman*)},leftmargin=*]
\item Section~\ref{sec:local} supplies local strong solvability, preservation of the simplex, and
positive-time positivity; Sections~\ref{sec:spectral}--\ref{sec:DG} then give a fixed separated
spectral rectangle and a uniform positive-time H\"older modulus for both roots.
\item Sections~\ref{sec:symmetrizer}--\ref{sec:first-schauder} freeze the mixing-entropy
Hessian on small state-diameter cylinders and obtain a uniform terminal $C^{1+a}$ estimate,
hence a spatial Lipschitz bound.
\item Section~\ref{sec:terminal-trace} freezes only the principal reduced operator
$-C(y')^{-1}\Delta$.  The $C^{1+a}$ bound makes this frozen family compact, the De Giorgi time
H\"older modulus controls its variation in $\mathcal L(X_1,X_0)$, and the spatial Lipschitz bound
controls the quadratic nondivergence forcing.  Uniform frozen maximal regularity and short-time
absorption then yield a uniform terminal bound in $X_{\gamma,p}=W_p^{2-2/p}$ with the Neumann
constraint.
\item Section~\ref{sec:restart} verifies the HMPW continuation hypotheses on a restarted orbit and
excludes a finite maximal time.  Zero-mass and coincident-friction cases reduce to scalar or
triangular parabolic equations and are treated separately there.
\end{enumerate}
The ternary-specific points are the sign-independent root geometry, the extension beyond positive
pair-sum generators, the treatment of the closed concentration simplex and repeated frictions,
and the terminal continuation argument in the resulting two-root coordinates.

\section{Model and algebraic reduction}\label{sec:model}

We use throughout the normalization corresponding to equal molar masses, unit total density,
and zero reaction.  Thus mass and molar fractions coincide.  These assumptions are hypotheses of Theorem~\ref{thm:main}, rather than consequences of a
change of units.  Let
\[
y=(y_1,y_2,y_3),\qquad y_i\ge0,\qquad y_1+y_2+y_3=1,
\]
and let the normalized Maxwell--Stefan diffusion equations be
\begin{align}
 \partial_t y_i+\Div J_i&=0,\qquad i=1,2,3,\\
 \nabla y_i&=\sum_{j\ne i}f_{ij}(y_iJ_j-y_jJ_i),\\
 \sum_{i=1}^3J_i&=0.
\end{align}
The boundary condition is
\[
J_i\cdot\nu=0.
\]
Consequently each species mass is conserved:
\begin{equation}\label{eq:species-mass-conservation}
\frac{\dd}{\dd t}\int_\Omega y_i(t,x)\,\dd x
=-\int_{\partial\Omega}J_i\cdot\nu\,\dd S=0,
\qquad i=1,2,3.
\end{equation}
In particular, a nonnegative component with zero initial mass vanishes identically.

Write
\[
\alpha=f_{12},\qquad \beta=f_{13},\qquad \gamma=f_{23}>0.
\]
When the three frictions are distinct, species are relabeled so that
\[
0<\alpha<\beta<\gamma.
\]

Set $y_3=1-y_1-y_2$ and $y'=(y_1,y_2)^{\sf T}$.  A direct substitution of $J_3=-J_1-J_2$ gives
\begin{align*}
-\nabla y_1
&=\bigl[\beta+(\alpha-\beta)y_2\bigr]J_1+(\beta-\alpha)y_1J_2,\\
-\nabla y_2
&=(\gamma-\alpha)y_2J_1+\bigl[\gamma+(\alpha-\gamma)y_1\bigr]J_2,
\end{align*}
so the reduced matrix relation below is exact, rather than a linearization.
Thus
\[
-\nabla y'=C(y')\binom{J_1}{J_2},
\]
where
\[
C(y')=
\begin{pmatrix}
\beta+(\alpha-\beta)y_2&(\beta-\alpha)y_1\\
(\gamma-\alpha)y_2&\gamma+(\alpha-\gamma)y_1
\end{pmatrix}.
\]
Expanding the determinant $\Pi$ and using $y_3=1-y_1-y_2$ gives
\[
\Pi=\alpha\beta y_1+\alpha\gamma y_2+\beta\gamma y_3,
\]
and therefore
\[
\Pi\ge \min\{\alpha\beta,\alpha\gamma,\beta\gamma\}>0
\]
on the entire closed simplex.  Hence
\begin{equation}\label{eq:reduced-system}
y'_t-\Div\bigl(C(y')^{-1}\nabla y'\bigr)=0.
\end{equation}
The reduced diffusion matrix is nonsingular up to the simplex boundary.  Since $C$ is invertible,
the no-flux condition is equivalent to
\[
\partial_\nu y'=0.
\]

\subsection{Core-diagonal Darken interpretation}\label{subsec:darken-model}
The preceding constant-friction system is exactly a ternary core-diagonal Darken system \cite[Section~9]{BotheDruet2023}.  The following calculation verifies this representation and identifies the spectral variables
used below.  With
$\Pi$ as above set
\begin{equation}\label{eq:ternary-Darken-tracers}
 d_1(y)=\frac{\gamma}{\Pi(y)},\qquad
 d_2(y)=\frac{\beta}{\Pi(y)},\qquad
 d_3(y)=\frac{\alpha}{\Pi(y)}.
\end{equation}
Then
\[
 S_D(y):=\sum_{i=1}^3\frac{y_i}{d_i(y)}
 =\Pi(y)\left(\frac{y_1}{\gamma}+\frac{y_2}{\beta}+\frac{y_3}{\alpha}\right)
 =\frac{\Pi(y)^2}{\alpha\beta\gamma}.
\]
The multicomponent Darken formula
\[
 \mathfrak D_{ij}=d_i d_j S_D
\]
therefore gives
\begin{equation}\label{eq:ternary-Darken-MS}
 \mathfrak D_{12}=\frac1\alpha,\qquad
 \mathfrak D_{13}=\frac1\beta,\qquad
 \mathfrak D_{23}=\frac1\gamma,
\end{equation}
which is exactly the present Maxwell--Stefan system since $f_{ij}=1/\mathfrak D_{ij}$.
This is the normalized molar form of the ternary core-diagonal representation in
\cite[Section~9]{BotheDruet2023}.

Equivalently, every positive constant ternary friction triple is multiplicatively separated.  Indeed,
\[
 b_1=\sqrt{\frac{\alpha\beta}{\gamma}},\qquad
 b_2=\sqrt{\frac{\alpha\gamma}{\beta}},\qquad
 b_3=\sqrt{\frac{\beta\gamma}{\alpha}}
\]
satisfies $f_{ij}=b_i b_j$.  Formula~\eqref{eq:ternary-Darken-tracers} also specifies the composition dependence of the
tracer diffusivities: the tracer vector has the fixed-ray form
\[
 d(y)=\Pi(y)^{-1}(\gamma,\beta,\alpha)^{\sf T}.
\]
After removal of the common scalar factor, the diagonal tracer mobility defines fixed Stieltjes
poles.  The scaling relation between its two Stieltjes roots and the trace--determinant roots is recorded in
Section~\ref{subsec:darken-bridge}.

\section{Local strong theory and positivity}\label{sec:local}
For $p>d+2$ the local theory of Herberg--Meyries--Pr\"uss--Wilke applies to the present
normalization.  Indeed, with equal molar masses $M=I$, total density $\rho=1$, and zero reaction,
their Maxwell--Stefan matrix satisfies, in every spatial direction $\alpha$,
\[
(B(y)J^\alpha)_i
=\sum_{j\ne i}f_{ij}(y_iJ_j^\alpha-y_jJ_i^\alpha)
=\partial_\alpha y_i.
\]
Moreover $\sum_i\partial_\alpha y_i=0$, so their projection
$P(y)=I-y\otimes\mathbf e$ acts as the identity on $\partial_\alpha y$.  The sign conventions agree under the reduction above: our matrix $C$ was defined by
$-\nabla y'=C(y')J'$, whereas the full Maxwell--Stefan matrix of
\cite{HerbergMeyriesPruessWilke2017} is defined by $B(y)J=\nabla y$.  After eliminating
$J_3=-J_1-J_2$, the first two rows of $B(y)J=\nabla y$ are exactly
$-C(y')J'=\nabla y'$.  Hence their inverse flux operator is represented by $-C(y')^{-1}$ in
these coordinates, and their equation with $+\Div(A(y)\nabla y)$ is precisely
\eqref{eq:reduced-system}.  Thus their quasilinear problem specializes exactly to the normalized
diffusion system considered here.

For compatible data in $W_p^{2-2/p}$ taking values in the closed simplex, the admissible
state neighborhood constructed in \cite[Section~3.1]{HerbergMeyriesPruessWilke2017} contains that
entire closed simplex; in particular, zero components at $t=0$ are allowed.  Then
\cite[Theorem~3.2]{HerbergMeyriesPruessWilke2017} yields a unique maximal strong solution on
$[0,T_{\max})$ and states lower semicontinuity of the maximal lifespan with respect to the trace
space.  If $\partial\Omega\in C^{2+\alpha_0}$, the solution is classical for every positive time by
\cite[Theorem~3.3]{HerbergMeyriesPruessWilke2017}.  Nonnegativity is preserved, and every
nontrivial species becomes strictly positive for every positive time by
\cite[Theorem~3.4]{HerbergMeyriesPruessWilke2017}.  The global a priori argument below uses the
maximal strong class, preservation/positivity, and the HMPW continuation framework.  Positive-time
classicality is retained as part of the conclusion of Theorem~\ref{thm:main}, but it is not needed
to justify the spectral truncation or De Giorgi tests.

\section{Trace--determinant and spectral coordinates}\label{sec:spectral}

Define
\begin{align}
\Lambda&=(\alpha+\beta)y_1+(\alpha+\gamma)y_2+(\beta+\gamma)y_3,\\
\Pi&=\alpha\beta y_1+\alpha\gamma y_2+\beta\gamma y_3,
\end{align}
and
\[
V=(\Lambda,\Pi)^{\sf T}.
\]
For pairwise distinct frictions, these variables form an affine coordinate system on the
concentration simplex.  Indeed, after eliminating $y_3$,
\[
D_{y'}V=L=
\begin{pmatrix}
\alpha-\gamma&\alpha-\beta\\
\beta(\alpha-\gamma)&\gamma(\alpha-\beta)
\end{pmatrix},
\qquad
\det L=(\alpha-\gamma)(\alpha-\beta)(\gamma-\beta)\ne0.
\]
Since $L$ is constant, the reduced equation gives
\[
V_t-\Div\bigl(LC(y')^{-1}L^{-1}\nabla V\bigr)=0.
\]
A direct $2\times2$ multiplication, using $\det C=\Pi$, yields the exact identity
\[
LC(y')^{-1}L^{-1}
=
\mathbb A(V):=
\begin{pmatrix}
\Lambda/\Pi&-1/\Pi\\
1&0
\end{pmatrix}.
\]
Thus
\[
V_t-\Div\bigl(\mathbb A(V)\nabla V\bigr)=0.
\]
For notational brevity we write
\[
A(V):=\mathbb A(V)
\]
from this point onward.  Equivalently,
\begin{align}
\Lambda_t-\Div\left(\frac{\Lambda\nabla\Lambda-\nabla\Pi}{\Pi}\right)&=0,\label{eq:Lambda}\\
\Pi_t-\Delta\Lambda&=0.\label{eq:Pi}
\end{align}

The physical $V$-state set is the compact triangle
\[
\mathfrak V=\operatorname{co}\bigl\{
(\alpha+\beta,\alpha\beta),
(\alpha+\gamma,\alpha\gamma),
(\beta+\gamma,\beta\gamma)
\bigr\}.
\]
The no-flux condition becomes
\[
\partial_\nu\Lambda=\partial_\nu\Pi=0,
\]
and for distinct frictions the converse also holds.

\subsection{Darken and Stieltjes root coordinates}\label{subsec:darken-bridge}
The Darken representation from Section~\ref{subsec:darken-model} identifies the two spectral
roots as Stieltjes roots.  Associate to the three species the fixed poles
\[
 \ell_1=\gamma,\qquad \ell_2=\beta,\qquad \ell_3=\alpha.
\]
The Stieltjes numerator for the fixed pole vector $\ell=(\gamma,\beta,\alpha)$ is
\begin{align}
 Q_y(z)
 &=y_1(\beta-z)(\alpha-z)
   +y_2(\gamma-z)(\alpha-z)
   +y_3(\gamma-z)(\beta-z) \notag\\
 &=z^2-\Lambda z+\Pi.\label{eq:darken-stieltjes-polynomial}
\end{align}
Consequently the roots $u,v$ introduced below from the trace and determinant are exactly the two
Stieltjes roots for the fixed poles $\ell_i$.  For the tracer diffusivities
$d_i(y)=\ell_i/\Pi(y)$, the numerator in a mobility variable $\zeta$ is instead
\begin{equation}\label{eq:darken-stieltjes-scaling}
 \sum_{i=1}^3 y_i\prod_{j\ne i}(d_j(y)-\zeta)
 =\Pi(y)^{-2}Q_y(\Pi(y)\zeta).
\end{equation}
Its roots are $u/\Pi=1/v$ and $v/\Pi=1/u$, since $\Pi=uv$.
These are the two diffusion eigenvalues; $u,v$ are the relaxation roots used in the PDE
calculations below.  In particular, after the ordering
$0<\alpha<\beta<\gamma$, their interlacing
\[
 \alpha<u<\beta<v<\gamma
\]
is the usual Stieltjes interlacing with the fixed poles, written in the reverse species order.
This identity also explains why the inverse map from $(u,v)$ to the concentrations is multi-affine
and why the Euler--Poisson--Darboux relation appears in the ternary calculation.

\subsection{Relation with additive moment and root coordinates}\label{subsec:additive-bridge}
This subsection gives the explicit coordinate relation with the additive-friction analysis of
~\cite{BotheAdditive2026}.  For $N=3$, this relation remains valid without positivity of the
pair-sum generators.  Define
\begin{equation}\label{eq:ternary-real-generators}
 g_1:=\frac{\alpha+\beta-\gamma}{2},\qquad
 g_2:=\frac{\alpha+\gamma-\beta}{2},\qquad
 g_3:=\frac{\beta+\gamma-\alpha}{2}.
\end{equation}
Then
\begin{equation}\label{eq:ternary-pairsum-representation}
 \alpha=g_1+g_2,\qquad
 \beta=g_1+g_3,\qquad
 \gamma=g_2+g_3,
\end{equation}
and this is the unique real pair-sum representation of the friction triple.  Moreover
\[
 g_2-g_1=\gamma-\beta>0,\qquad
 g_3-g_2=\beta-\alpha>0,
\]
so $g_1<g_2<g_3$.  The generators are all positive if and only if the strict triangle
inequalities hold for $(\alpha,\beta,\gamma)$; under the present ordering only
$\gamma<\alpha+\beta$ is not implied by the ordering.  Thus arbitrary positive ternary frictions are always
pair-additive over \emph{real} generators, but need not belong to the positive-generator
constitutive class considered in~\cite{BotheAdditive2026}.

Set
\[
 e_1:=g_1+g_2+g_3,\qquad
 e_2:=g_1g_2+g_1g_3+g_2g_3,
\]
and introduce the first two additive moments
\[
 U_1:=\sum_{i=1}^3 g_i y_i,\qquad
 U_2:=\sum_{i=1}^3 g_i^2y_i.
\]
The trace--determinant coordinates above are exactly fixed translates of these moments:
\begin{equation}\label{eq:ternary-moment-translation}
 \Lambda=e_1+U_1,\qquad \Pi=e_2+U_2.
\end{equation}
Indeed,
\[
 e_1+g_1=\alpha+\beta,\qquad
 e_1+g_2=\alpha+\gamma,\qquad
 e_1+g_3=\beta+\gamma,
\]
while
\[
 e_2+g_1^2=\alpha\beta,\qquad
 e_2+g_2^2=\alpha\gamma,\qquad
 e_2+g_3^2=\beta\gamma.
\]

For an interior composition define the additive resolvent
\[
 m_y(z):=\sum_{i=1}^3\frac{y_i}{g_i-z}
\]
and its monic numerator
\[
 q_y(z):=\prod_{i=1}^3(g_i-z)\,m_y(z)
 =\sum_{i=1}^3 y_i\prod_{j\ne i}(g_j-z).
\]
Since $y_i>0$ and $g_1<g_2<g_3$, the usual sign change across the poles gives two simple roots
\[
 z_1\in(g_1,g_2),\qquad z_2\in(g_2,g_3).
\]
Expanding the numerator and using \eqref{eq:ternary-moment-translation} gives
\begin{align*}
 q_y(z)
 &=z^2-(e_1-U_1)z+e_2-e_1U_1+U_2,\\
 q_y(e_1-\lambda)
 &=\lambda^2-(e_1+U_1)\lambda+(e_2+U_2)
 =p(\lambda).
\end{align*}
Consequently the two spectral systems are related by the exact affine transformation
\begin{equation}\label{eq:ternary-additive-root-transform}
 u=e_1-z_2,\qquad v=e_1-z_1.
\end{equation}
In particular,
\[
 g_1<z_1<g_2<z_2<g_3
 \quad\Longleftrightarrow\quad
 \alpha<u<\beta<v<\gamma.
\]
Thus the ternary trace--determinant and root coordinates coincide with the two-root additive
moment and root coordinates expressed in the positive relaxation variables $u,v$.  The present theorem also applies when the unique real generator $g_1$ is nonpositive.  The same
formulas describe the coincident-friction cases: equality of two pair frictions is equivalent, after
the corresponding relabeling, to equality of two real generators; in this case the interlacing
root chart degenerates and the triangular scalar reduction of Section~\ref{sec:restart} applies.

\subsection{Spectral roots and invariant rectangles}

Let
\[
p(\lambda)=\lambda^2-\Lambda\lambda+\Pi
\]
and denote the two roots by
\[
u:=\lambda_1,
\qquad
v:=\lambda_2,
\qquad u<v.
\]
The ordering of the roots follows directly from the simplex variables.  Substitution of the
barycentric formulas for $\Lambda$ and $\Pi$ gives
\begin{equation}\label{eq:polynomial-endpoint-signs}
\begin{aligned}
p(\alpha)&=(\beta-\alpha)(\gamma-\alpha)y_3,\\
p(\beta)&=-(\beta-\alpha)(\gamma-\beta)y_2,\\
p(\gamma)&=(\gamma-\alpha)(\gamma-\beta)y_1.
\end{aligned}
\end{equation}
Hence for an interior state $p(\alpha)>0$, $p(\beta)<0$, and $p(\gamma)>0$.  Since $p$ is a
monic quadratic, its two simple real roots satisfy
\[
\alpha<u<\beta<v<\gamma.
\]
Moreover
\[
\Lambda=u+v,
\qquad
\Pi=uv.
\]
Evaluating $p(t)=(t-u)(t-v)$ in \eqref{eq:polynomial-endpoint-signs} gives the inverse composition map
\begin{align}
y_1&=\frac{(\gamma-u)(\gamma-v)}{(\gamma-\alpha)(\gamma-\beta)},\\
y_2&=\frac{(\beta-u)(v-\beta)}{(\beta-\alpha)(\gamma-\beta)},\\
y_3&=\frac{(u-\alpha)(v-\alpha)}{(\beta-\alpha)(\gamma-\alpha)}.
\end{align}
Since $y'\mapsto V$ is an invertible affine map and
$\det D_{(u,v)}(u+v,uv)=u-v\neq0$, these formulas show that the open simplex is analytically
diffeomorphic to
\[
(\alpha,\beta)\times(\beta,\gamma).
\]
The Neumann condition passes exactly to the roots on this separated region.  Indeed,
\[
\partial_\nu\Lambda=\partial_\nu u+\partial_\nu v=0,
\qquad
\partial_\nu\Pi=v\,\partial_\nu u+u\,\partial_\nu v=0.
\]
Since $u\neq v$, this $2\times2$ linear system gives
\begin{equation}\label{eq:root-neumann}
\partial_\nu u=\partial_\nu v=0.
\end{equation}
This observation will also be used below to eliminate the boundary entropy flux.

The scalar spectral equations may be written in drift form, and a direct positive-part energy
comparison gives forward invariance of the coordinate ranges in the maximal-$L^p$ strong class.
The precise terminal-strip statement used later is the following.

\begin{lemma}[Forward-invariant separated spectral rectangle]\label{lem:forward-spectral-rectangle}
Let $y$ be the maximal strong solution, let $t_*>0$, and suppose
$y_i(t_*,x)>0$ for every species and every $x\in\overline\Omega$.  Define
\[
 a_*:=\min_{\overline\Omega}u(t_*),\qquad
 b_*:=\max_{\overline\Omega}u(t_*),\qquad
 c_*:=\min_{\overline\Omega}v(t_*),\qquad
 d_*:=\max_{\overline\Omega}v(t_*).
\]
Then
\[
 \alpha<a_*\le b_*<\beta<c_*\le d_*<\gamma
\]
and
\begin{equation}\label{eq:forward-spectral-rectangle}
 a_*\le u(t,x)\le b_*,\qquad
 c_*\le v(t,x)\le d_*
\end{equation}
for every $(t,x)\in[t_*,T_{\max})\times\overline\Omega$.  Moreover, if
\[
 \underline u(t):=\min_{\overline\Omega}u(t),\qquad
 \overline u(t):=\max_{\overline\Omega}u(t)
\]
and similarly for $v$, then for $t_*\le t_1\le t_2<T_{\max}$,
\begin{equation}\label{eq:root-extrema-monotonicity-ternary}
 \underline u(t_1)\le\underline u(t_2)\le\overline u(t_2)\le\overline u(t_1),
 \qquad
 \underline v(t_1)\le\underline v(t_2)\le\overline v(t_2)\le\overline v(t_1).
\end{equation}
In particular, both coordinatewise root oscillations are nonincreasing.  Consequently, if all
three conserved species masses are positive, then for every $\tau>0$ one compact separated
rectangle contains the spectral range on the whole tail
$[\tau,T_{\max})\times\overline\Omega$.
\end{lemma}

\begin{proof}
Strict positivity at $t=t_*$ and the inverse formulas above imply pointwise
$\alpha<u<\beta<v<\gamma$.  Compactness of $\overline\Omega$ therefore gives the strict endpoint
separation in the statement.

We prove forward invariance directly in the maximal-$L^p$ strong class.  Fix
$T\in(t_*,T_{\max})$ and choose $\delta>0$ so small that
\[
 \mathcal R_\delta
 :=[a_*-\delta,b_*+\delta]\times[c_*-\delta,d_*+\delta]
 \Subset(\alpha,\beta)\times(\beta,\gamma).
\]
Let $T_\delta\in(t_*,T]$ be the supremum of the times $s$ for which the ordered roots are
defined on $[t_*,s]$ and their range lies in $\mathcal R_\delta$, with $T_\delta=T$ if this
holds up to $T$.  Since $\mathcal R_\delta$ is compactly contained in the open root chart and the
initial range is a distance $\delta$ from its boundary, continuity gives $T_\delta>t_*$.  Suppose
that $T_\delta<T$ and fix $T'\in(t_*,T_\delta)$.

On $[t_*,T']$ the root map is smooth on a neighborhood of $\mathcal R_\delta$.  The maximal
strong regularity of $y$, the affine relation between $y'$ and $V$, and the Sobolev/Bochner chain
rule therefore give
\begin{equation}\label{eq:root-strong-class-rectangle}
 (u,v)\in W^{1,p}(t_*,T';L^p(\Omega;\mathbb R^2))
 \cap L^p(t_*,T';W^{2,p}(\Omega;\mathbb R^2))
 \cap C([t_*,T'];C^1(\overline\Omega;\mathbb R^2)).
\end{equation}
The root equations hold distributionally and $\partial_\nu u=\partial_\nu v=0$ in the trace
sense.  In divergence form they read
\begin{align*}
 u_t-\Div\!\left(\frac1v\nabla u\right)+W_u\cdot\nabla u&=0,
 &W_u&=-h(u,v)\nabla v,\\
 v_t-\Div\!\left(\frac1u\nabla v\right)+W_v\cdot\nabla v&=0,
 &W_v&= h(u,v)\nabla u.
\end{align*}
On $\mathcal R_\delta$, the scalar diffusion coefficients are uniformly positive and $h$ is
bounded; by \eqref{eq:root-strong-class-rectangle}, $W_u,W_v$ are bounded on the strip.

Put $M_u=b_*+\delta/2$ and $r=(u-M_u)_+$.  Testing the first equation by $r$ and using the
homogeneous Neumann trace gives, for almost every time,
\[
 \frac12\frac{\dd}{\dd t}\|r\|_2^2
 +a_0\|\nabla r\|_2^2
 \le \|W_u\|_\infty\int_\Omega r|\nabla r|
 \le \frac{a_0}{2}\|\nabla r\|_2^2+C\|r\|_2^2,
\]
where $a_0>0$ depends only on $\mathcal R_\delta$.  Since $r(t_*)=0$, Gronwall's lemma gives
$r\equiv0$.  Applying the same argument to $(a_*-\delta/2-u)_+$, and then to the corresponding
upper and lower truncations of $v$, yields
\[
 a_*-\frac\delta2\le u\le b_*+\frac\delta2,
 \qquad
 c_*-\frac\delta2\le v\le d_*+\frac\delta2
 \quad\text{on }[t_*,T']\times\overline\Omega.
\]
All tests are admissible in \eqref{eq:root-strong-class-rectangle}: one may use the standard
Lipschitz positive-part approximation, and the time-slice identity follows from the usual
Bochner chain rule.  No classical comparison theorem is used.

Letting $T'\uparrow T_\delta$ and using continuity of $V$ shows that the limiting state lies in
the trace--determinant image of the strictly smaller rectangle obtained by replacing $\delta$ with
$\delta/2$.  That image is compactly contained in the open root chart.  Hence continuity of $V$
keeps the state inside the chart and inside $\mathcal R_\delta$ for a short time beyond
$T_\delta$, contradicting the definition of $T_\delta$.  Thus $T_\delta=T$.  Since $T<T_{\max}$ was
arbitrary and the conclusion holds for every sufficiently small $\delta$, letting
$\delta\downarrow0$ proves \eqref{eq:forward-spectral-rectangle} on the whole preterminal tail.

For \eqref{eq:root-extrema-monotonicity-ternary}, restart the same strong-class invariant-rectangle
argument at an arbitrary time $t_1\ge t_*$, using the extrema
$\underline u(t_1),\overline u(t_1),\underline v(t_1),\overline v(t_1)$ as the rectangle endpoints.
They remain strictly inside the interlacing intervals by \eqref{eq:forward-spectral-rectangle}.
The restarted comparison gives the four monotonicity inequalities at every later time $t_2$.

If all conserved masses are positive, no component is identically zero and the Maxwell--Stefan
positivity theorem gives $y_i(t,x)>0$ for every $t>0$; see
\cite[Theorem~3.4]{HerbergMeyriesPruessWilke2017}.  Applying the first part at any
$t_*<\tau$ gives the final tail statement.
\end{proof}

\begin{remark}[Scope of spectral coordinates]
The $V$-system is smooth on the entire closed physical triangle because $\Pi$ has a fixed positive
lower bound.  Quantities containing $(v-u)^{-1}$, however, should be used only on a positive-time
compact spectral rectangle (or under an a priori interior-range hypothesis).
\end{remark}
\subsection{Exact scalar spectral equations}

The scalar equations follow directly from the two equations for $(\Lambda,\Pi)$.  Since
$\Lambda=u+v$ and $\Pi=uv$,
\[
\Lambda\nabla\Lambda-\nabla\Pi
=u\nabla u+v\nabla v,
\]
and therefore \eqref{eq:Lambda}--\eqref{eq:Pi} become
\begin{align*}
u_t+v_t
&=\Div\left(\frac1v\nabla u+\frac1u\nabla v\right),\\
v u_t+u v_t&=\Delta u+\Delta v.
\end{align*}
Expanding only the first divergence and solving this $2\times2$ linear system for $u_t,v_t$
yields the capacity equations
\begin{align}
v\,u_t-\Delta u
=\mathcal H(u,v)\,\nabla u\cdot\nabla v,\label{eq:capacity-u}\\
u\,v_t-\Delta v
=-\mathcal H(u,v)\,\nabla u\cdot\nabla v,\label{eq:capacity-v}
\end{align}
with
\[
\mathcal H(u,v)=\frac{u^2+v^2}{uv(v-u)}.
\]
Rewriting these two equations in divergence form gives
\begin{align}
u_t-\Div\left(\frac1v\nabla u\right)
&=h(u,v)\,\nabla u\cdot\nabla v,\label{eq:spec-div-u}\\
v_t-\Div\left(\frac1u\nabla v\right)
&=-h(u,v)\,\nabla u\cdot\nabla v,\label{eq:spec-div-v}
\end{align}
where
\[
h(u,v)=\frac{u+v}{uv(v-u)}.
\]
The coefficient also admits the decomposition
\[
\mathcal H(u,v)=\frac{2}{v-u}+\frac1u-\frac1v.
\]

On every compact positive-time strip on which the root range lies in a separated rectangle,
these identities hold at the maximal-$L^p$ strong-solution level without assuming classical
regularity.  Indeed, the root map is smooth on such a rectangle; hence the affine strong regularity
of $V$ and the Sobolev/Bochner chain rules give
\begin{equation}\label{eq:strong-root-class}
 (u,v)\in W^{1,p}(I;L^p(\Omega;\mathbb R^2))
 \cap L^p(I;W^{2,p}(\Omega;\mathbb R^2))
 \cap C(\overline I;C^1(\overline\Omega;\mathbb R^2)),
\end{equation}
where the last embedding uses $p>d+2$.  The preceding algebra may therefore be interpreted
distributionally, and \eqref{eq:root-neumann} holds in the Sobolev trace sense.  This observation
will be used below to justify all nonlinear root tests without appealing to positive-time
classical Schauder regularity.

The equations \eqref{eq:capacity-u}--\eqref{eq:capacity-v} have the following properties:
\begin{itemize}
\item the spatial principal operator is exactly the scalar Laplacian in both equations;
\item the only principal time coefficients are the positive capacities $v$ and $u$;
\item there are no $|\nabla u|^2$ or $|\nabla v|^2$ source terms;
\item the only quadratic source is the antisymmetric mixed product $\nabla u\cdot\nabla v$.
\end{itemize}
Adding the capacity equations gives
\[
v u_t+u v_t-\Delta(u+v)=0,
\]
hence the exact cancellation
\[
(uv)_t-\Delta(u+v)=0,
\]
which is precisely \eqref{eq:Pi}.

On a fixed compact spectral rectangle,
\begin{align}
|\mathcal H(u,v)\nabla u\cdot\nabla v|&\le C|\nabla u|\,|\nabla(u,v)|,\\
|\mathcal H(u,v)\nabla u\cdot\nabla v|&\le C|\nabla v|\,|\nabla(u,v)|.
\end{align}
Thus the mixed term can be estimated with either $|\nabla u|$ or $|\nabla v|$ as the selected
gradient factor.

\section{Spectral entropies and stabilized one-sided truncations}\label{sec:entropy}

Fix a compact separated spectral rectangle
\begin{equation}\label{eq:sep-rectangle}
\alpha<u_-\le u\le u_+<\beta<v_-\le v\le v_+<\gamma,
\end{equation}
write
\[
\mathcal R:=[u_-,u_+]\times[v_-,v_+],
\qquad
\delta:=v_--u_+>0,
\]
and note that $\mathcal R$ has positive distance from all spectral collision and endpoint sets.
By the inverse composition formulas, compactness of $\mathcal R$ also gives a constant
$m_{\mathcal R}>0$ such that
\[
 y_i(u,v)\ge m_{\mathcal R},\qquad i=1,2,3,\quad (u,v)\in\mathcal R.
\]
In particular, the mixing entropy and all of its derivatives used below are uniformly bounded on a
fixed neighborhood of $\mathcal R$.  All constants below depend only on $\mathcal R$ and the
fixed ternary frictions.

\subsection{The spectral entropy-production operator}

Write
\[
a_1=\frac1v,\qquad a_2=\frac1u.
\]
For any $F\in C^2$ of the spectral variables, multiplication of the spectral equations by
$F_u$ and $F_v$ gives
\begin{equation}\label{eq:scalar-test-identity-stab}
\partial_tF-\Div\left(a_1F_u\nabla u+a_2F_v\nabla v\right)+\mathscr Q[F]=0,
\end{equation}
where
\begin{equation}\label{eq:Q-def-stab}
\begin{aligned}
\mathscr Q[F]
={}&a_1F_{uu}|\nabla u|^2+a_2F_{vv}|\nabla v|^2\\
&+(a_1+a_2)\left(
F_{uv}-\frac{F_u-F_v}{v-u}
\right)\nabla u\cdot\nabla v.
\end{aligned}
\end{equation}
Hence the Euler--Poisson--Darboux relation
\begin{equation}\label{eq:EPD-stab}
F_{uv}=\frac{F_u-F_v}{v-u}
\end{equation}
cancels all mixed quadratic production.

The relation with the multi-root EPD system introduced in~\cite{BotheAdditive2026} is exact.
Accordingly, the EPD identity is the two-root form of the identity from the additive-friction
analysis, rather than a separate ternary identity.  The required separated root geometry must still
be verified for arbitrary positive ternary frictions.  Under
\eqref{eq:ternary-additive-root-transform}, write
$F(u,v)=\Psi(e_1-v,e_1-u)$.  Then
\[
 (z_2-z_1)\Psi_{z_1z_2}=\Psi_{z_1}-\Psi_{z_2}
 \quad\Longleftrightarrow\quad
 F_{uv}=\frac{F_u-F_v}{v-u}.
\]
Thus \eqref{eq:EPD-stab} is precisely the two-root specialization of the multi-EPD equations,
expressed in the positive relaxation variables.

\subsection{The ordinary mixing entropy as an EPD entropy}

Let
\[
\ell_1=\gamma,\qquad \ell_2=\beta,\qquad \ell_3=\alpha,
\qquad
c_i=\frac1{\prod_{j\ne i}(\ell_i-\ell_j)}.
\]
The inverse spectral map is
\begin{equation}\label{eq:lagrange-inverse-y}
y_i=c_i(\ell_i-u)(\ell_i-v),\qquad i=1,2,3.
\end{equation}
For
\[
\mathfrak h(u,v)=\sum_{i=1}^3 y_i(u,v)\log y_i(u,v)
\]
put $f_i=y_i\log y_i$.  Since
\[
(y_i)_u=-c_i(\ell_i-v),\qquad
(y_i)_v=-c_i(\ell_i-u),\qquad
(y_i)_{uv}=c_i,
\]
we have $(y_i)_u-(y_i)_v=c_i(v-u)$ and
$(y_i)_u(y_i)_v/y_i=c_i$.  Hence
\[
(f_i)_{uv}
=c_i(\log y_i+2),\qquad
\frac{(f_i)_u-(f_i)_v}{v-u}=c_i(\log y_i+1),
\]
which proves
\[
(f_i)_{uv}-\frac{(f_i)_u-(f_i)_v}{v-u}=c_i.
\]
Since the barycentric weights satisfy $\sum_i c_i=0$,
\begin{equation}\label{eq:mixing-entropy-EPD}
\mathfrak h_{uv}=\frac{\mathfrak h_u-\mathfrak h_v}{v-u}.
\end{equation}
Thus the physical mixing entropy belongs exactly to the spectral EPD family.

If
\[
P(t)=(t-\alpha)(t-\beta)(t-\gamma),
\]
then the partial-fraction identity
\[
\frac1{P(t)}=\sum_{i=1}^3\frac{c_i}{t-\ell_i}
\]
gives
\[
\mathfrak h_{uu}=\frac{v-u}{P(u)},
\qquad
\mathfrak h_{vv}=-\frac{v-u}{P(v)}.
\]
Therefore
\begin{equation}\label{eq:mixing-spectral-production}
\begin{aligned}
\mathscr Q[\mathfrak h]
={}&
\frac{v-u}{v(u-\alpha)(\beta-u)(\gamma-u)}|\nabla u|^2\\
&+
\frac{v-u}{u(v-\alpha)(v-\beta)(\gamma-v)}|\nabla v|^2.
\end{aligned}
\end{equation}
In particular,
\begin{equation}\label{eq:mixing-coercive}
\mathscr Q[\mathfrak h]\ge
\kappa_{\mathcal R}\bigl(|\nabla u|^2+|\nabla v|^2\bigr)
\end{equation}
on \eqref{eq:sep-rectangle}.

\subsection{EPD entropies are Hessian symmetrizers}

The EPD identity also implies a Hessian symmetrization property used in the continuation argument.  Let
$F(V)=\Phi(u,v)$, where $V=(\Lambda,\Pi)=(u+v,uv)$, and write
\[
J=\begin{pmatrix}1&1\\ v&u\end{pmatrix}
=\bigl[V_u,V_v\bigr].
\]
The columns are eigenvectors of the trace--determinant diffusion matrix:
\begin{equation}\label{eq:AJ-spectral}
A(V)J=J\begin{pmatrix}v^{-1}&0\\0&u^{-1}\end{pmatrix}.
\end{equation}
If $\Phi$ satisfies the EPD equation, then
\[
\Phi_{uv}=\frac{\Phi_u-\Phi_v}{v-u}=F_\Pi.
\]
On the other hand the chain rule gives
\[
\Phi_{uv}=V_u^{\sf T}D_V^2F(V)V_v+F_\Pi.
\]
Hence
\begin{equation}\label{eq:H-orthogonality}
V_u^{\sf T}D_V^2F(V)V_v=0.
\end{equation}
Since $V_{uu}=V_{vv}=0$,
\[
V_u^{\sf T}D_V^2F(V)V_u=\Phi_{uu},\qquad
V_v^{\sf T}D_V^2F(V)V_v=\Phi_{vv}.
\]
Therefore
\begin{equation}\label{eq:H-spectral-diag}
J^{\sf T}D_V^2F(V)J
=\begin{pmatrix}\Phi_{uu}&0\\0&\Phi_{vv}\end{pmatrix}.
\end{equation}
Combining \eqref{eq:AJ-spectral} and \eqref{eq:H-spectral-diag} gives the exact Hessian
symmetrizer identity
\begin{equation}\label{eq:Hessian-symmetrizer}
D_V^2F(V)A(V)=A(V)^{\sf T}D_V^2F(V).
\end{equation}
Thus the EPD equation is equivalent, for separated roots, to orthogonality of the two diffusion
eigendirections in the Hessian metric of $F$.

For the ordinary mixing entropy $F=\mathfrak h$, the explicit spectral Hessian coefficients are
\[
\mathfrak h_{uu}
=\frac{v-u}{(u-\alpha)(\beta-u)(\gamma-u)}>0,
\]
\[
\mathfrak h_{vv}
=\frac{v-u}{(v-\alpha)(v-\beta)(\gamma-v)}>0.
\]
Consequently $H(V):=D_V^2\mathfrak h(V)$ is positive definite.  Moreover,
using \eqref{eq:AJ-spectral} and \eqref{eq:H-spectral-diag},
\[
J^{\sf T}H(V)A(V)J
=\begin{pmatrix}\mathfrak h_{uu}/v&0\\0&\mathfrak h_{vv}/u\end{pmatrix}>0.
\]
Since $J$ is invertible for $u<v$, this proves
\begin{equation}\label{eq:mixing-H-positive}
H(V)>0,\qquad H(V)A(V)=A(V)^{\sf T}H(V)>0
\end{equation}
on every compact separated spectral rectangle.  This gives a canonical state-dependent
Friedrichs symmetrizer for arbitrary positive ternary frictions.

\subsection{Pure \texorpdfstring{$C^{1,1}$}{C1,1} one-sided EPD truncations}

For $k\in[u_-,u_+]$ set
\[
P_{u,k}(u,v)
=\frac12(v-k)(u-k)^2-\frac16(u-k)^3.
\]
Then
\[
(P_{u,k})_{uu}=v-u,\qquad
(P_{u,k})_{vv}=0,\qquad
(P_{u,k})_{uv}
=\frac{(P_{u,k})_u-(P_{u,k})_v}{v-u},
\]
and $P_{u,k}=DP_{u,k}=0$ on $u=k$.  Define
\[
\Phi_{u,k}^+=\mathbf1_{\{u>k\}}P_{u,k},
\qquad
\Phi_{u,k}^-=\mathbf1_{\{u<k\}}P_{u,k}.
\]
Equivalently, with $s_+=\max\{s,0\}$,
\begin{align*}
\Phi_{u,k}^+
&=\frac12(v-k)(u-k)_+^2-\frac16(u-k)_+^3,\\
\Phi_{u,k}^-
&=\frac12(v-k)(k-u)_+^2+\frac16(k-u)_+^3.
\end{align*}
In particular their first derivatives are globally Lipschitz across $u=k$.
These functions are $C^{1,1}$ on the compact state rectangle.  We use their explicit weak
derivatives rather than invoking a generic $C^{1,1}$ chain rule.  On the active side write $F=P_{u,k}$ and on the inactive
side write $0$.  Since
\[
F=0,\qquad DF=0\qquad\text{on }u=k,
\]
one has the explicit globally Lipschitz first derivative
\[
D\Phi_{u,k}^{\pm}=\mathbf1_{\{\pm(u-k)>0\}}DF
\]
and its weak target derivative is exactly
\begin{equation}\label{eq:weak-Hessian-trunc}
D(D\Phi_{u,k}^{\pm})
=\mathbf1_{\{\pm(u-k)>0\}}D^2F
\quad\text{a.e.},
\end{equation}
with no hypersurface measure on $u=k$.  After composition with a separated maximal-$L^p$ strong root solution, the precise level-set fact used is
Stampacchia's identity
\[
\nabla u=0\qquad\text{a.e. on }\{u=k\}.
\]
Together with $F=DF=0$ on the target interface, this shows directly that the composed weak
Hessian has the active-side value and no interface measure.  Thus the active-side Hessian formula
is valid almost everywhere.  Generic state-space mollification is not used, since it would in
general destroy the exact EPD identity.

For $r=(u-k)_+$,
\[
\Phi_{u,k}^+
=r^2\left(\frac{v-u}{2}+\frac r3\right),
\]
whereas for $r=(k-u)_+$,
\[
\Phi_{u,k}^-
=r^2\left(\frac{v-u}{2}-\frac r3\right).
\]
For the minus case, the lower bound follows from the fact that, on its active set,
$r=k-u<v-u$ because $k\le u_+<v_-\le v$.  Hence
\[
\frac{v-u}{2}-\frac r3\ge\frac{v-u}{6}\ge\frac\delta6,
\]
while the plus coefficient is at least $(v-u)/2\ge\delta/2$.  Compactness of the rectangle gives
matching upper bounds.  Thus for both signs
\begin{equation}\label{eq:Phi-u-comparable}
c_0r^2\le\Phi_{u,k}^{\pm}\le C_0r^2,
\end{equation}
and the EPD cancellation yields
\begin{equation}\label{eq:pure-u-production}
\mathscr Q[\Phi_{u,k}^{\pm}]
=
\frac{v-u}{v}|\nabla r|^2
\ge c_1|\nabla r|^2.
\end{equation}
Furthermore
\begin{equation}\label{eq:pure-u-derivative-size}
|(\Phi_{u,k}^{\pm})_u|\le Cr,\qquad
|(\Phi_{u,k}^{\pm})_v|\le Cr^2.
\end{equation}

For $k\in[v_-,v_+]$ the symmetric upper-root polynomial
\[
P_{v,k}(u,v)
=\frac12(k-u)(v-k)^2+\frac16(v-k)^3
\]
gives one-sided $C^{1,1}$ truncations $\Phi_{v,k}^{\pm}$.  The same interval comparison gives
\[
c_0r^2\le \Phi_{v,k}^{\pm}\le C_0r^2,
\]
with constants uniform in $k$: for the active plus truncation $r=v-k<v-u$ because $k>u$, while
for the active minus truncation the coefficient is larger.  Moreover
\begin{equation}\label{eq:pure-v-production}
\mathscr Q[\Phi_{v,k}^{\pm}]
=
\frac{v-u}{u}|\nabla(v-k)_\pm|^2,
\end{equation}
and the selected first derivative has size $O(r)$ whereas the opposite first derivative has
size $O(r^2)$.

\begin{remark}[Two-root specialization of the multi-EPD truncations]
Under the affine transformation \eqref{eq:ternary-additive-root-transform}, the general additive
one-sided root polynomial
\[
 (-1)^{i-1}\int_k^{z_i}(z_i-\sigma)\prod_{j\ne i}(z_j-\sigma)\,\dd\sigma
\]
from~\cite{BotheAdditive2026}, with two root variables, becomes after the corresponding affine
level change exactly one of the cubic polynomials $P_{u,k}$ or $P_{v,k}$ above.  Thus the pure ternary truncations coincide with the $N=3$ specialization of the multi-EPD family
in the relaxation-root coordinates.  The mixing-entropy product used below is a
stabilization of this exact EPD truncation and is not itself an exact EPD solution.
\end{remark}

\subsection{Entropy-stabilized truncations}

Choose $C_{\mathcal R}$ so that
\[
\widehat{\mathfrak h}:=\mathfrak h+C_{\mathcal R}\ge1
\]
on the compact rectangle.  Put $q=r^2/2$ and define
\begin{equation}\label{eq:stabilized-u-entropy}
\Psi_{u,k}^{\pm}
=
\Phi_{u,k}^{\pm}
+\varepsilon q\,\widehat{\mathfrak h}.
\end{equation}
The same construction is used for $v$.  Again $\Psi\in C^{1,1}$ and
\begin{equation}\label{eq:stabilized-comparable}
cr^2\le\Psi\le Cr^2.
\end{equation}

Because $q=q(u)$, the product expansion is exact:
\begin{equation}\label{eq:product-Q-exact}
\begin{aligned}
\mathscr Q[q\widehat{\mathfrak h}]
={}&
q\,\mathscr Q[\mathfrak h]
+\widehat{\mathfrak h}\,\mathscr Q[q]
+2a_1q_u\mathfrak h_u|\nabla u|^2\\
&+(a_1+a_2)q_u\mathfrak h_v\,\nabla u\cdot\nabla v,
\end{aligned}
\end{equation}
where
\begin{equation}\label{eq:Qq-exact}
\mathscr Q[q]
=
a_1|\nabla r|^2
-\frac{a_1+a_2}{v-u}\,q_u\,\nabla u\cdot\nabla v,
\qquad |q_u|=r.
\end{equation}
The term $q\mathscr Q[\mathfrak h]$ supplies
\[
\frac{\kappa_{\mathcal R}}2r^2
\bigl(|\nabla u|^2+|\nabla v|^2\bigr).
\]
All possibly negative additional terms are bounded by
\[
C_1\varepsilon r|\nabla r|^2
+C_2\varepsilon r|\nabla r|\,|\nabla v|.
\]
The absorption constants can be chosen explicitly.  Let
$R_0=\sup r$ over all admissible levels on the fixed rectangle, let $c_1>0$ be the constant in
\eqref{eq:pure-u-production}, and let $\kappa_{\mathcal R}>0$ be the constant in
\eqref{eq:mixing-coercive}.  The mixed error satisfies
\[
C_2\varepsilon r|\nabla r|\,|\nabla v|
\le \frac{c_1}{4}|\nabla r|^2
+\frac{C_2^2}{c_1}\varepsilon^2r^2|\nabla v|^2.
\]
Fix
\begin{equation}\label{eq:epsilon-stabilization-choice}
0<\varepsilon_{\mathcal R}\le
\min\left\{1,\frac{c_1}{4C_1R_0},
\frac{c_1\kappa_{\mathcal R}}{4C_2^2}\right\},
\end{equation}
where the corresponding fraction is omitted if $C_1R_0=0$ or $C_2=0$.  Then the first error is bounded by
$c_1|\nabla r|^2/4$, the selected-gradient term from Young's inequality is bounded by the same
quantity, and
\[
\frac{C_2^2}{c_1}\varepsilon_{\mathcal R}^2r^2|\nabla v|^2
\le \frac{\kappa_{\mathcal R}}4\varepsilon_{\mathcal R}
r^2|\nabla v|^2
\]
is absorbed by the opposite-gradient part of
$\varepsilon q\mathscr Q[\mathfrak h]$, whose coefficient is
$\kappa_{\mathcal R}\varepsilon r^2/2$.  The remaining positive terms may be discarded.
The upper-root calculation is identical after interchanging the selected and opposite spectral
variables.  Decrease $\varepsilon_{\mathcal R}$, if necessary, to the minimum of the two admissible
choices.  This fixed parameter is used for both roots, both signs, and all admissible levels.
We have therefore proved the following lemma.
\begin{lemma}[Stabilized scalar coercivity]\label{lem:stabilized-coercivity}
For one fixed $\varepsilon_{\mathcal R}>0$, every lower-root level and either sign satisfy
\begin{equation}\label{eq:stabilized-u-coercivity}
\mathscr Q[\Psi_{u,k}^{\pm}]
\ge
c_{\mathcal R}|\nabla r|^2
+c_{\mathcal R}\varepsilon_{\mathcal R}r^2|\nabla v|^2.
\end{equation}
For upper-root truncations,
\begin{equation}\label{eq:stabilized-v-coercivity}
\mathscr Q[\Psi_{v,k}^{\pm}]
\ge
c_{\mathcal R}|\nabla r|^2
+c_{\mathcal R}\varepsilon_{\mathcal R}r^2|\nabla u|^2.
\end{equation}
\end{lemma}

The corresponding entropy flux
\[
J_\Psi=\frac{\Psi_u}{v}\nabla u+\frac{\Psi_v}{u}\nabla v
\]
obeys, for a lower-root truncation,
\begin{equation}\label{eq:stabilized-flux-size}
|J_\Psi|
\le C_{\mathcal R}
\bigl(r|\nabla r|+r^2|\nabla v|\bigr),
\end{equation}
and symmetrically for $v$.  The two terms in the flux estimate correspond to the two positive terms
in \eqref{eq:stabilized-u-coercivity}; no integration by parts in the opposite spectral gradient
is needed.

\begin{lemma}[Strong root chain rule, including one-sided spectral entropies]
\label{lem:C11-entropy-chain}
Let $I=(t_-,t_+)$ be an interval with $\overline I\subset(0,T_{\max})$.  Assume that the root
range lies in $\mathcal R$ on $I$ and that $(u,v)$ has the strong regularity
\eqref{eq:strong-root-class}, satisfies
\eqref{eq:spec-div-u}--\eqref{eq:spec-div-v} distributionally, and has
$\partial_\nu u=\partial_\nu v=0$ in the Sobolev trace sense.

\emph{(i)} If $F\in C^2$ on a neighborhood of $\mathcal R$, then
\begin{equation}\label{eq:C11-entropy-distribution}
 \partial_tF-\Div\!\left(\frac{F_u}{v}\nabla u+
 \frac{F_v}{u}\nabla v\right)+\mathscr Q[F]=0
\end{equation}
holds in $\mathcal D'(I\times\Omega)$.

\emph{(ii)} Fix one selected root $\xi\in\{u,v\}$, a level $k$ in its invariant interval, and a
smooth branch function $G$ satisfying
\begin{equation}\label{eq:one-sided-branch-vanishing-ternary}
 G=0,\qquad DG=0
 \quad\text{on the target interface }\{\xi=k\}.
\end{equation}
Define either
\[
 \Psi^+=\mathbf 1_{\{\xi>k\}}G,
 \qquad\text{or}\qquad
 \Psi^-=\mathbf 1_{\{\xi<k\}}G.
\]
Then $\Psi^\pm\in C^{1,1}$, its weak target Hessian equals $D^2G$ on the active side and zero on
the inactive side, and \eqref{eq:C11-entropy-distribution} holds with $F$ replaced by $\Psi^\pm$
and with this a.e. Hessian in $\mathscr Q[\Psi^\pm]$.  No hypersurface measure is created on the
truncation interface.  In particular this applies to every pure and stabilized one-sided entropy
constructed above.

In both parts the entropy flux belongs to $L^p(I;W^{1,p}(\Omega;\mathbb R^d))$, has zero normal
trace, and for every bounded spatial cutoff $\zeta$ the map
\[
 t\longmapsto\int_\Omega\zeta(x)F(u(t,x),v(t,x))\,\dd x
\]
(and analogously with $\Psi^\pm$) is absolutely continuous.  Hence the entropy identity may be
integrated between arbitrary time slices in $\overline I$.
\end{lemma}

\begin{proof}
Because $p>d+2$, \eqref{eq:strong-root-class} gives bounded first spatial derivatives on the
closed strip.  All state coefficients and derivatives occurring below are bounded on the compact
rectangle $\mathcal R$.  Thus $v^{-1}\nabla u,u^{-1}\nabla v\in L^p(I;W^{1,p})$, the quadratic
root terms are bounded, and every term in the root equations is represented by an $L^p$ function.

For $F\in C^2$, the Bochner chain rule gives
\[
 \partial_tF=F_u u_t+F_v v_t
 \quad\text{in }L^p(I\times\Omega),
\]
while the Sobolev spatial chain rule gives
\[
 \nabla F_u=F_{uu}\nabla u+F_{uv}\nabla v,
 \qquad
 \nabla F_v=F_{uv}\nabla u+F_{vv}\nabla v.
\]
Multiply \eqref{eq:spec-div-u} and \eqref{eq:spec-div-v} by $F_u$ and $F_v$, respectively, use
the distributional product rule, and add.  The diagonal Hessian terms give
$v^{-1}F_{uu}|\nabla u|^2+u^{-1}F_{vv}|\nabla v|^2$.  The two mixed Hessian terms contribute
$(u^{-1}+v^{-1})F_{uv}\nabla u\cdot\nabla v$, while the explicit root couplings contribute
\[
 -h(F_u-F_v)\nabla u\cdot\nabla v
 =-\left(\frac1u+\frac1v\right)
   \frac{F_u-F_v}{v-u}\nabla u\cdot\nabla v.
\]
This is exactly \eqref{eq:Q-def-stab}, proving part~(i).

For part~(ii), in target space
\[
 D\Psi^\pm=\mathbf 1_{\rm active}DG.
\]
Because $DG=0$ on the interface and $D^2G$ is bounded, this vector field is globally Lipschitz and
its weak target derivative is
\[
 D(D\Psi^\pm)=\mathbf 1_{\rm active}D^2G.
\]
There is no surface measure because the trace of $D\Psi^\pm$ has no jump.  The Lipschitz Sobolev
chain rule therefore applies to $D\Psi^\pm(u,v)$.  On the preimage of the target interface all
components of $D\Psi^\pm(u,v)$ vanish.  Stampacchia's zero-set property, applied componentwise,
then gives the same active-side Hessian formula for the spatial weak derivative almost everywhere.
The first-order Bochner time chain rule uses only that $\Psi^\pm\in C^1$ with bounded derivative.
Repeating the algebra from part~(i) proves the distributional identity.  This is the strong-class
counterpart of the identities for positive and negative parts; no classical
approximation of the nonlinear root system and no state-space mollification of the entropy are
used.

Finally, differentiating the entropy flux spatially produces only $D^2u,D^2v\in L^p$ and products
of bounded first derivatives, so the flux belongs to $L^p(I;W^{1,p})$.  Its normal trace vanishes
because $\partial_\nu u=\partial_\nu v=0$.  The composed entropy belongs to
$W^{1,p}(I;L^p(\Omega))$ with time derivative given by the chain rule; multiplying by a bounded
spatial cutoff and integrating in $x$ gives absolute continuity in time and hence integration up to
arbitrary time slices.
\end{proof}

\section{De Giorgi estimates up to the boundary}\label{sec:DG}

The De Giorgi argument in this section specializes the regularity method developed
in~\cite{BotheAdditive2026} to two roots.  We include the estimates because the unrestricted
ternary setting requires verification of the constants, boundary-cylinder estimates, and later
continuation hypotheses independently of the signs of the pair-sum generators.  No methodological
novelty is claimed for the abstract iteration.  The additional point proved here is that the exact
two-root EPD truncations and their entropy stabilization remain available without positivity of
these generators.

In addition to the quadratic Caccioppoli estimate, the entropy-stabilized truncations built from the
one-sided EPD entropies yield a weighted gradient estimate for the opposite root, with the squared
selected truncation as weight.  This additional estimate is used in the De Giorgi argument below.  We therefore do not identify the roots with an abstract
normalized De Giorgi class.  Instead we derive the
oscillation contraction from four quantitative estimates: quadratic Caccioppoli control,
a logarithmic time-propagation estimate for the scalar root equation, De Giorgi shrinking, and the
critical-mass iteration.

For $x_0\in\overline\Omega$ and $\rho>0$ put
\[
D_\rho(x_0):=B_\rho(x_0)\cap\Omega,
\qquad
Q_\rho^\Omega(t_0,x_0):=(t_0-\rho^2,t_0)\times D_\rho(x_0).
\]
We suppress the centers when no confusion can arise.

\begin{lemma}[Uniform local geometry up to the physical boundary]
\label{lem:uniform-local-geometry-DG}
There are $\rho_\Omega>0$ and constants depending only on $d$ and the $C^2$ geometry of
$\Omega$ such that, for every $x_0\in\overline\Omega$ and $0<\rho\le\rho_\Omega$:
\begin{enumerate}[label=(\roman*)]
\item $|D_\rho(x_0)|\asymp_\Omega\rho^d$.  Moreover, for every $\varepsilon>0$ there is
$\lambda_\varepsilon\in(0,1)$, uniform in $x_0$ and $\rho$, such that
\[
 |D_\rho(x_0)\setminus D_{\lambda_\varepsilon\rho}(x_0)|
 \le \varepsilon |D_\rho(x_0)|;
\]
\item the parabolic Sobolev estimate
\begin{equation}\label{eq:boundary-parabolic-sobolev-DG}
 \iint_{I\times D_\rho}|q|^{2(d+2)/d}
 \le C_\Omega
 \left(\esssup_{t\in I}\int_{D_\rho}|q(t)|^2\right)^{2/d}
 \iint_{I\times D_\rho}|\nabla q|^2
\end{equation}
holds for functions $q$ which vanish near the artificial boundary
$\partial B_\rho(x_0)\cap\Omega$, with no condition on the physical boundary (with the usual
one-dimensional interpretation when $d=1$);
\item for $w\in W^{1,1}(D_\rho)$ and $h<\ell$,
\begin{equation}\label{eq:relative-DeGiorgi-isoperimetric}
 (\ell-h)
 |\{w\le h\}\cap D_\rho|
 |\{w\ge \ell\}\cap D_\rho|
 \le C_\Omega\rho^{d+1}
 \int_{D_\rho\cap\{h<w<\ell\}}|\nabla w|\,\dd x.
\end{equation}
\end{enumerate}
\end{lemma}

\begin{proof}
For $d=1$, sufficiently small caps are intervals, and the assertions follow from
the one-dimensional Sobolev and variation inequalities.  Suppose $d\ge2$.
Fix a finite $C^2$ boundary atlas, a tubular radius, and a common bound $K$ for the
second derivatives of the boundary graphs.  Decrease $\rho_\Omega$ so that every cap meeting
the boundary lies in a graph chart.  At a nearest boundary point, use coordinates in which
\[
 \Omega=\{x_d>\psi(x')\},\quad \psi(0)=0,\quad D\psi(0)=0,
 \quad x_0=(0,t),\quad 0\le t\le\rho,
\]
and $\|D^2\psi\|\le K$.  Put
\[
 A_\rho:=B_{\rho/16}(x_0+(\rho/4)e_d).
\]
For $K\rho$ sufficiently small, $A_\rho\subset D_\rho$ and every $a\in A_\rho$
has height $a_d-\psi(a')\ge\rho/8$ above the graph.  Moreover, $D_\rho$ is star-shaped
with respect to every point of $A_\rho$.  To verify this, let $z\in D_\rho$ and $0\le s\le1$.
The Hessian bound gives
\begin{align*}
 &(1-s)a_d+sz_d-\psi((1-s)a'+sz')\\
 &\quad\ge (1-s)(a_d-\psi(a'))+s(z_d-\psi(z'))
       -\frac K2s(1-s)|a'-z'|^2>0
\end{align*}
for $s<1$, after decreasing $\rho_\Omega$ once, because $|a'-z'|\le2\rho$.
At $s=1$ the point is $z\in\Omega$.  The segment also stays in $B_\rho(x_0)$ by convexity
of that ball.  For interior balls the same conclusions hold with $A_\rho=B_{\rho/16}(x_0)$.
Thus every cap has diameter at most $2\rho$ and contains an averaging ball of radius
$\rho/16$ with this segment property, uniformly in the center.

This proves the lower measure bound in (i).  It also gives the required uniform $L^1$
Poincar\'e inequality directly.  For smooth $f$, integrate
\[
 |f(z)-f(a)|\le 2\rho\int_0^1
 |\nabla f((1-s)a+sz)|\,\dd s
\]
over $z\in D_\rho$ and $a\in A_\rho$.  For $s\le1/2$ change variables in $a$;
for $s\ge1/2$ change variables in $z$.  The respective Jacobian factors are at most
$2^d$, and all image points remain in $D_\rho$.  Division by $|A_\rho|$ yields
\[
 \int_{D_\rho}|f-f_{A_\rho}|\le C_d\rho\int_{D_\rho}|\nabla f|,
 \qquad f_{A_\rho}:=|A_\rho|^{-1}\int_{A_\rho}f.
\]
Sobolev density extends the estimate to $W^{1,1}$.  Apply it to
$f=\min\{(w-h)_+,\ell-h\}$.  The double integral of $|f(x)-f(y)|$ over
$D_\rho\times D_\rho$ is at least
$2(\ell-h)|\{w\le h\}||\{w\ge\ell\}|$ and at most
$2|D_\rho|\int_{D_\rho}|f-f_{A_\rho}|$.  This proves (iii).
It supplies the cap estimate directly, without requiring the curved caps to be convex;
compare the convex-domain uniform inequalities in~\cite{Thomas2015}.

For (ii), extend a function vanishing near the artificial spherical boundary by zero outside
the cap within $\Omega$.  In a boundary chart flatten the physical boundary and reflect this
function evenly across its flat face.  The function is supported in a ball of radius $C\rho$;
the chart and its inverse have uniform Jacobian and derivative bounds.  The Euclidean
Gagliardo--Nirenberg inequality for this extension gives, at each time,
\[
 \int_{D_\rho}|q|^{2(d+2)/d}
 \le C_\Omega\left(\int_{D_\rho}|q|^2\right)^{2/d}
                    \int_{D_\rho}|\nabla q|^2.
\]
Interior caps use zero extension only.  Integrating in time proves (ii).
Finally the elementary annulus bound
\[
 |D_\rho\setminus D_{\lambda\rho}|
 \le \omega_d(1-\lambda^d)\rho^d
 \le C_\Omega(1-\lambda^d)|D_\rho|
\]
proves the remaining assertion in (i).  Only test functions are extended or reflected;
the Maxwell--Stefan system itself is not reflected.
\end{proof}

\subsection{Quadratic and two-time stabilized estimates}

\begin{proposition}[Stabilized one-sided root Caccioppoli inequalities]
\label{prop:DG-Caccioppoli}
Assume the separated spectral rectangle \eqref{eq:sep-rectangle}.  There are constants
$c_{\mathcal R},C_{\mathcal R}>0$ such that the following holds uniformly on every interior or
boundary cap below $\rho_\Omega$.  For $k\in[u_-,u_+]$, either sign, and
$r=(u-k)_\pm$, every $C^1$ cutoff $0\le\eta\le1$ which vanishes near the lower time face and the
artificial lateral boundary, but may be nonzero on $\partial\Omega$, satisfies
\begin{align}
 &\esssup_t\int \eta^2r^2(t)
 +c_{\mathcal R}\iint\eta^2|\nabla r|^2
 +c_{\mathcal R}\varepsilon_{\mathcal R}\iint
   \eta^2r^2|\nabla v|^2 \notag\\
 &\qquad\le
 C_{\mathcal R}\iint r^2
 \bigl(|\nabla\eta|^2+\eta|\partial_t\eta|\bigr).
 \label{eq:DG-cutoff-with-opposite}
\end{align}
Consequently, for $0<\rho<R\le\rho_\Omega$,
\begin{align}
 &\esssup_{t\in(t_0-\rho^2,t_0)}
   \int_{D_\rho}r^2(t)
 +\iint_{Q_\rho^\Omega}|\nabla r|^2
 +\varepsilon_{\mathcal R}\iint_{Q_\rho^\Omega}r^2|\nabla v|^2
 \notag\\
 &\qquad\le
 \frac{C_{\mathcal R}}{(R-\rho)^2}
 \iint_{Q_R^\Omega}r^2.
 \label{eq:DG-nested-with-opposite}
\end{align}
The same statements hold for $k\in[v_-,v_+]$, with $r=(v-k)_\pm$ and $u$ in place of the
opposite root.
\end{proposition}

\begin{proof}
By Lemma~\ref{lem:C11-entropy-chain}, the stabilized $C^{1,1}$ entropy
$\Psi=\Psi_{u,k}^{\pm}$ satisfies the distributional identity
\[
 \partial_t\Psi-\Div J_\Psi+\mathscr Q[\Psi]=0.
\]
Because $\partial_\nu u=\partial_\nu v=0$, one has $J_\Psi\cdot\nu=0$ on the physical boundary.
Multiplication by $\eta^2$, integration up to an arbitrary time, the flux estimate
\eqref{eq:stabilized-flux-size}, and the coercivity
\eqref{eq:stabilized-u-coercivity} give, after Young's inequality,
\[
 \esssup_t\int\eta^2\Psi(t)
 +c\iint\eta^2|\nabla r|^2
 +c\varepsilon_{\mathcal R}\iint\eta^2r^2|\nabla v|^2
 \le C\iint r^2(|\nabla\eta|^2+\eta|\partial_t\eta|).
\]
Now use $\Psi\asymp_{\mathcal R}r^2$.  Standard nested parabolic cutoffs give
\eqref{eq:DG-nested-with-opposite}.  The upper-root calculation is identical.
\end{proof}

\begin{lemma}[Two-time stabilized level energy]\label{lem:two-time-stabilized-energy}
Under the hypotheses of Proposition~\ref{prop:DG-Caccioppoli}, let $s<t$ and let
$\zeta=\zeta(x)$ be a time-independent spatial cutoff which vanishes near the artificial boundary
and may be nonzero on $\partial\Omega$.  For $r=(u-k)_\pm$ and
$\Psi=\Psi_{u,k}^{\pm}$,
\begin{align}
 &\int \zeta^2\Psi(t)
 +c_{\mathcal R}\int_s^t\!\!\int\zeta^2|\nabla r|^2
 +c_{\mathcal R}\varepsilon_{\mathcal R}\int_s^t\!\!\int
   \zeta^2r^2|\nabla v|^2 \notag\\
 &\qquad\le
 \int\zeta^2\Psi(s)
 +C_{\mathcal R}\int_s^t\!\!\int r^2|\nabla\zeta|^2.
 \label{eq:two-time-stabilized-energy}
\end{align}
The symmetric statement holds for upper-root truncations.
\end{lemma}

\begin{proof}
Integrate the exact stabilized entropy identity between $s$ and $t$ and localize only in space.
There is therefore no temporal cutoff and the lower-time entropy appears with coefficient one.  The
same flux absorption as in Proposition~\ref{prop:DG-Caccioppoli} yields
\eqref{eq:two-time-stabilized-energy}.
\end{proof}

\subsection{A logarithmic estimate from the scalar root equations}

For the lower root, \eqref{eq:spec-div-u} is
\begin{equation}\label{eq:selected-u-scalar-drift}
 u_t-\Div\!\left(\frac1v\nabla u\right)-h(u,v)\nabla u\cdot\nabla v=0,
 \qquad \partial_\nu u=0,
\end{equation}
and for the upper root
\begin{equation}\label{eq:selected-v-scalar-drift}
 v_t-\Div\!\left(\frac1u\nabla v\right)+h(u,v)\nabla v\cdot\nabla u=0,
 \qquad \partial_\nu v=0.
\end{equation}
On the fixed separated rectangle, $u^{-1}$, $v^{-1}$, and $h(u,v)$ are uniformly bounded, while
$u^{-1}$ and $v^{-1}$ are uniformly positive.  We shall use the following identity under a change of orientation.  If $w$ is one of the roots, $z$ is the opposite root, and its scalar equation is written
\[
 w_t-\Div(a(w,z)\nabla w)+c(w,z)\nabla w\cdot\nabla z=0,
 \qquad \partial_\nu w=0,
\]
then, for $\varepsilon\in\{-1,1\}$ and a constant $c_0$,
$\widetilde w=\varepsilon w+c_0$ satisfies
\begin{equation}\label{eq:oriented-ternary-root-equation}
 \widetilde w_t-\Div(a(w,z)\nabla\widetilde w)
 +c(w,z)\nabla\widetilde w\cdot\nabla z=0,
 \qquad \partial_\nu\widetilde w=0.
\end{equation}
The coefficient field is still evaluated along the original state.  Thus an upper-phase logarithmic
test for $\widetilde w$ is the lower-phase test for $w$ when $\varepsilon=-1$; no
symmetry of the coefficients is being assumed.

Fix $\sigma_0\in(0,1/2)$, say $\sigma_0=1/4$.  For $H>0$,
$0<\delta<\sigma_0/2$, and $0\le r\le H$, define
\begin{align}
 \ell_{H,\delta}(r)
 &:={\left[\log\frac{(1-\sigma_0+\delta)H}
 {H-r+\delta H}\right]}_+,
 \label{eq:scaled-log-ell}\\
 F_{H,\delta}(r)&:=H^2\ell_{H,\delta}(r)^2.
 \label{eq:scaled-log-F}
\end{align}
Then $F_{H,\delta}=0$ on $[0,\sigma_0H]$.  On its active interval, with
$D(r)=H-r+\delta H$,
\[
 F'_{H,\delta}(r)=\frac{2H^2\ell_{H,\delta}(r)}{D(r)},
 \qquad
 F''_{H,\delta}(r)=\frac{2H^2(1+\ell_{H,\delta}(r))}{D(r)^2}.
\]
Thus, with
\[
 L_\delta:=\log\frac{1-\sigma_0+\delta}{\delta},
\]
\begin{equation}\label{eq:scaled-log-ratio}
 \frac{(F'_{H,\delta}(r))^2}{F''_{H,\delta}(r)}
 \le 2\sigma_0^{-2}L_\delta r^2,
\end{equation}
where the quotient is defined as zero on the inactive interval.  Estimate
\eqref{eq:scaled-log-ratio} shows that the drift-generated weighted square is controlled by the
stabilized entropy.

\begin{lemma}[Scaled logarithmic-square root estimate]\label{lem:root-logarithmic-estimate}
Let $w$ denote either spectral root, let $z$ denote the opposite root, and let
$r=(w-k)_+$ or $r=(k-w)_+$ satisfy $0\le r\le H$ on the cylinder under consideration.  The level
$k$ is assumed to lie in the corresponding trapped root interval.  Let $\eta,\zeta$ be
 time-independent spatial cutoffs with $0\le\eta,\zeta\le1$,
$\zeta\equiv1$ on a neighborhood of $\operatorname{supp}\eta$, both vanishing near the artificial
spatial boundary but unrestricted on $\partial\Omega$.  Then
\begin{align}
 &\int\eta^2F_{H,\delta}(r(t))
 +c_{\mathcal R}\int_s^t\!\!\int
   \eta^2F''_{H,\delta}(r)|\nabla r|^2 \notag\\
 &\quad\le
 \int\eta^2F_{H,\delta}(r(s))
 +C_{\mathcal R}L_\delta\int\zeta^2r^2(s) \notag\\
 &\qquad
 +C_{\mathcal R}L_\delta\int_s^t\!\!\int
 r^2\bigl(|\nabla\eta|^2+|\nabla\zeta|^2\bigr).
 \label{eq:root-logarithmic-estimate}
\end{align}
The constants are independent of $H,\delta,s,t$ and are uniform for interior and boundary caps.
\end{lemma}

\begin{proof}
We write the proof for $w=u$ and $r=(u-k)_+$; the other cases differ only by orientation and by
interchanging $u,v$.  On the active set, test \eqref{eq:selected-u-scalar-drift} with
$\eta^2F'_{H,\delta}(r)$ and integrate between $s$ and $t$.  The principal part gives
\[
 \int_s^t\!\!\int \frac1v\eta^2F''_{H,\delta}(r)|\nabla r|^2.
\]
The localization and drift errors are bounded by
\[
 C_{\mathcal R}\eta|F'_{H,\delta}(r)|\,|\nabla r|\,|\nabla\eta|,
 \qquad
 C_{\mathcal R}\eta^2|F'_{H,\delta}(r)|\,|\nabla r|\,|\nabla v|.
\]
Young's inequality, the positive lower bound for $1/v$, and
\eqref{eq:scaled-log-ratio} give
\begin{align*}
 &\int\eta^2F_{H,\delta}(r(t))
 +c\int_s^t\!\!\int\eta^2F''_{H,\delta}(r)|\nabla r|^2\\
 &\qquad\le
 \int\eta^2F_{H,\delta}(r(s))
 +C_{\mathcal R}L_\delta\int_s^t\!\!\int r^2|\nabla\eta|^2
 +C_{\mathcal R}L_\delta\int_s^t\!\!\int\eta^2r^2|\nabla v|^2.
\end{align*}
Apply Lemma~\ref{lem:two-time-stabilized-energy} at the same level and sign with cutoff $\zeta$.
Because $\zeta=1$ on $\operatorname{supp}\eta$ and $\Psi\asymp_{\mathcal R}r^2$,
\[
 \int_s^t\!\!\int\eta^2r^2|\nabla v|^2
 \le C_{\mathcal R}
 \left(
  \int\zeta^2r^2(s)
  +\int_s^t\!\!\int r^2|\nabla\zeta|^2
 \right).
\]
This proves \eqref{eq:root-logarithmic-estimate}.  The scalar test does not require a formal
nonsmooth chain rule.  In the plus case set
\[
 \widehat F_{H,\delta}(r):=
 \begin{cases}
 0,&r<0,\\
 F_{H,\delta}(r),&0\le r\le H,\\
 F_{H,\delta}(H)+F'_{H,\delta}(H)(r-H),&r>H,
 \end{cases}
 \qquad G(s):=\widehat F_{H,\delta}((s-k)_+).
\]
In the minus case set $G(s):=\widehat F_{H,\delta}((k-s)_+)$.
This extension is needed because the original logarithmic formula is defined here only on
$[0,H]$ and would become singular at $r=(1+\delta)H$ if continued unchanged.
The first derivatives match at $r=H$, while $F_{H,\delta}$ vanishes on
$[0,\sigma_0H]$.  Hence $G\in C^{1,1}(\mathbb R)$ and $G'$ is bounded and Lipschitz
for each fixed $H,\delta$; it vanishes near the original truncation level $s=k$.
There can be jumps in the a.e. second derivative at the logarithmic activation level and at
$r=H$, but $G'$ remains continuous.  Along the solution, $0\le r\le H$,
so all estimates above are unchanged; on $\{r=H\}$ its spatial gradient vanishes a.e.  Thus $\eta^2G'(w)$ is an admissible $W^{1,2}$ spatial test on every compact
preterminal slab (or is obtained by smooth density), and
\[
 \nabla G'(w)=G''(w)\nabla w\qquad\text{a.e.}
\]
with no singular measure; on a level set where $G''$ is undefined, Stampacchia's identity
$\nabla w=0$ a.e. makes its assigned value irrelevant.  The ordinary $C^1$ time chain rule gives
$\partial_tG(w)=G'(w)w_t$.  This justifies the displayed calculation in the weak formulation.
At the physical boundary $\partial_\nu w=0$, so no boundary term is produced.
\end{proof}

\begin{lemma}[Short-time propagation of level-set measure]\label{lem:root-measure-propagation}
There exist structural numbers
$\delta_0,\lambda_0,\epsilon_0,\gamma_0\in(0,1)$ with the following property.  Let
$0<\rho\le\rho_\Omega$, let $w$ be either root with invariant interval $[a,b]$, and suppose
$[k,k+H]\subset[a,b]$ and
\[
 w\le k+H
 \quad\text{on }[s,s+\delta_0\rho^2]\times D_\rho.
\]
If
\[
 |\{w(s)\le k\}\cap D_\rho|\ge\frac12|D_\rho|,
\]
then, for every $t\in[s,s+\delta_0\rho^2]$,
\begin{equation}\label{eq:propagated-low-fraction}
 |\{w(t)\le k+(1-\epsilon_0)H\}\cap D_{\lambda_0\rho}|
 \ge\gamma_0|D_{\lambda_0\rho}|.
\end{equation}
The symmetric upper-phase statement also holds.
\end{lemma}

\begin{proof}
The case $H=0$ is immediate, so assume $H>0$.
Fix $\sigma_0=1/4$.  By Lemma~\ref{lem:uniform-local-geometry-DG}, choose
$\lambda_0\in(0,1)$ so close to one that
\begin{equation}\label{eq:thin-annulus-choice-scaled-log}
 |D_\rho\setminus D_{\lambda_0\rho}|\le\frac1{16}|D_\rho|.
\end{equation}
Choose radii $\lambda_0<\lambda_1<\lambda_2<1$ and cutoffs $\eta,\zeta$ with
$\eta=1$ on $D_{\lambda_0\rho}$, $\operatorname{supp}\eta\subset D_{\lambda_1\rho}$,
$\zeta=1$ on $D_{\lambda_1\rho}$, $\operatorname{supp}\zeta\subset D_{\lambda_2\rho}$, and
\[
 |\nabla\eta|+|\nabla\zeta|\le \frac{C}{(1-\lambda_0)\rho}.
\]
Set $r=(w-k)_+$.  Then $0\le r\le H$ and the initial half-measure assumption gives
\begin{equation}\label{eq:scaled-log-initial-bound}
 \int\eta^2F_{H,\delta}(r(s))\le\frac12H^2L_\delta^2|D_\rho|,
 \qquad
 \int\zeta^2r^2(s)\le\frac12H^2|D_\rho|.
\end{equation}
Put
\[
 L_\delta^-:=\log\frac{1-\sigma_0+\delta}{2\delta}=L_\delta-\log2.
\]
If $r>(1-\delta)H$, then
$F_{H,\delta}(r)\ge H^2(L_\delta^-)^2$.  Lemma~\ref{lem:root-logarithmic-estimate} therefore gives,
for $t\in[s,s+\delta_0\rho^2]$,
\begin{align}
 &H^2(L_\delta^-)^2
 |\{r(t)>(1-\delta)H\}\cap D_{\lambda_0\rho}| \notag\\
 &\quad\le
 H^2\left[
 \frac12L_\delta^2
 +C_{\mathcal R}L_\delta
 \left(1+\frac{\delta_0}{(1-\lambda_0)^2}\right)
 \right]|D_\rho|.
 \label{eq:scaled-log-propagation-ratio}
\end{align}
Set
\[
 A_0:=C_{\mathcal R}\left(1+\frac{1}{(1-\lambda_0)^2}\right).
\]
Choose $\epsilon_0\in(0,\sigma_0/2)$ so small that
\begin{equation}
 \frac{\frac12L_{\epsilon_0}^2+A_0L_{\epsilon_0}}
 {(L_{\epsilon_0}^-)^2}
 \le \frac{9}{16}.
 \label{eq:choose-epsilon-log-propagation}
\end{equation}
This is possible because $L_\delta/L_\delta^-\to1$ and
$L_\delta/(L_\delta^-)^2\to0$ as $\delta\downarrow0$.
Then fix
\[
 \delta_0:=\frac14\min\{1,\lambda_0^2\}.
\]
Since $\delta_0\le1$, the bracket in \eqref{eq:scaled-log-propagation-ratio} is bounded by
$\frac12L_{\epsilon_0}^2+A_0L_{\epsilon_0}$.
Using \eqref{eq:thin-annulus-choice-scaled-log},
$|D_{\lambda_0\rho}|\ge15|D_\rho|/16$, hence the relative measure of
$\{r>(1-\epsilon_0)H\}$ in $D_{\lambda_0\rho}$ is at most $3/5$.  Thus
\eqref{eq:propagated-low-fraction} holds, for example, with $\gamma_0=2/5$.  The opposite phase is
obtained from the lower truncation and the oppositely oriented scalar test.
\end{proof}

\subsection{Critical mass, shrinking, and oscillation decay}

\begin{lemma}[Critical mass from quadratic Caccioppoli]\label{lem:root-critical-mass}
There is a structural $\nu_*>0$ such that the following holds.  Let
$0<\rho\le\rho_\Omega/2$ and let $Q_{2\rho}^\Omega(t_0,x_0)$ lie in a positive-time strip.
Let $w$ be either root with trapped interval $[a,b]$, and suppose $[k,k+H]\subset[a,b]$,
$0\le(w-k)_+\le H$, and
\[
 \frac{|\{w>k\}\cap Q_{2\rho}^\Omega|}{|Q_{2\rho}^\Omega|}\le\nu_*.
\]
Then
\begin{equation}\label{eq:critical-mass-conclusion}
 \esssup_{Q_\rho^\Omega}(w-k)_+\le\frac H2.
\end{equation}
The analogous lower statement holds.
\end{lemma}

\begin{proof}
The case $H=0$ is immediate, so assume $H>0$.
Put
\[
 \rho_n=\rho(1+2^{-n}),
 \qquad
 k_n=k+\frac H2(1-2^{-n}),
 \qquad
 A_n=\{w>k_n\}\cap Q_{\rho_n}^\Omega.
\]
Choose parabolic cutoffs $\eta_n$ supported in $Q_{\rho_n}^\Omega$, equal to one on
$Q_{\rho_{n+1}}^\Omega$, vanishing at the lower and artificial lateral boundaries, with
$|\nabla\eta_n|\le C2^n/\rho$ and
$|\partial_t\eta_n|\le C4^n/\rho^2$.  Proposition~\ref{prop:DG-Caccioppoli}, applied to
$w_n=(w-k_n)_+$, gives
\[
 E_n:=\esssup_t\int(w_n\eta_n)^2
      +\iint|\nabla(w_n\eta_n)|^2
 \le C_{\mathcal R}\frac{4^n}{\rho^2}H^2|A_n|.
\]
With $q=2(d+2)/d$, \eqref{eq:boundary-parabolic-sobolev-DG} gives
$\iint|w_n\eta_n|^q\le C E_n^{1+2/d}$.  On $A_{n+1}$,
$w_n\ge H2^{-n-2}$, hence
\[
 |A_{n+1}|
 \le C b^n\rho^{-2(d+2)/d}|A_n|^{1+2/d}
\]
for one structural $b>1$.  Since $|Q_{2\rho}^\Omega|\asymp_\Omega\rho^{d+2}$, the normalized
measures $Y_n=|A_n|/|Q_{2\rho}^\Omega|$ satisfy
\[
 Y_{n+1}\le Cb^nY_n^{1+2/d}.
\]
Set $\delta_D=2/d$.  The elementary induction
\[
 Y_n\le C^{-1/\delta_D}b^{-1/\delta_D^2-n/\delta_D}
\]
is valid whenever $Y_0\le C^{-1/\delta_D}b^{-1/\delta_D^2}$.  Choose $\nu_*>0$ no larger
than this structural threshold.  Then $Y_n\to0$, and the limiting level is $k+H/2$, proving
\eqref{eq:critical-mass-conclusion}.
\end{proof}

\begin{lemma}[Shrinking under a persistent opposite phase]\label{lem:root-shrinking}
Fix $\alpha,\delta\in(0,1)$ and $\Lambda>1$.  Let
$I=(s,s+\delta\rho^2)$ and suppose
\begin{equation}\label{eq:persistent-low-phase}
 |\{w(t)\le k_0\}\cap D_\rho|\ge\alpha|D_\rho|
 \qquad(t\in I).
\end{equation}
Assume $\Lambda\rho\le\rho_\Omega$ and that $D_{\Lambda\rho}$ lies in the same local chart,
$[k_0,k_0+M]\subset[a,b]$, and
$w\le k_0+M$ on $I\times D_{\Lambda\rho}$.  Define
$k_j=k_0+M(1-2^{-j})$.  Then
\begin{equation}\label{eq:shrinking-measure-estimate}
 \frac{|\{w>k_J\}\cap(I\times D_\rho)|}{|I\times D_\rho|}
 \le\frac{C_{\mathcal R}(\Lambda)}{\alpha\sqrt{\delta J}},
 \qquad J\ge1.
\end{equation}
The symmetric statement holds with signs reversed.
\end{lemma}

\begin{proof}
The case $M=0$ is immediate, so assume $M>0$.
Set
$A_j(t)=\{w(t)>k_j\}\cap D_\rho$ and
$E_j(t)=\{k_j<w(t)<k_{j+1}\}\cap D_\rho$.  Since
$\{w\le k_0\}\subset\{w\le k_j\}$, the relative isoperimetric inequality gives
\[
 (k_{j+1}-k_j)|A_{j+1}(t)|
 \le\frac{C_\Omega\rho}{\alpha}
 \int_{E_j(t)}|\nabla w|\,\dd x.
\]
Integrating in time and using Cauchy--Schwarz,
\begin{equation}\label{eq:shrinking-pre-sum}
 (k_{j+1}-k_j)|A_{j+1}|
 \le\frac{C_\Omega\rho}{\alpha}
 \left(\iint_{I\times D_\rho}|\nabla(w-k_j)_+|^2\right)^{1/2}|E_j|^{1/2}.
\end{equation}
Choose a spatial cutoff $\zeta=1$ on $D_\rho$, vanishing near the artificial boundary of
$D_{\Lambda\rho}$, with $|\nabla\zeta|\le C_\Lambda/\rho$.  Since
$0\le(w-k_j)_+\le M2^{-j}$ on the buffered cylinder, Lemma~\ref{lem:two-time-stabilized-energy}
gives
\begin{equation}\label{eq:shrinking-gradient-bound-buffered}
 \iint_{I\times D_\rho}|\nabla(w-k_j)_+|^2
 \le C_{\mathcal R}(\Lambda)M^24^{-j}\rho^d.
\end{equation}
As $k_{j+1}-k_j=M2^{-j-1}$,
\eqref{eq:shrinking-pre-sum}--\eqref{eq:shrinking-gradient-bound-buffered} imply
\[
 |A_{j+1}|^2\le\frac{C_{\mathcal R}(\Lambda)}{\alpha^2}\rho^{d+2}|E_j|.
\]
The $E_j$ are pairwise disjoint and $A_J\subset A_{j+1}$ for $j<J$.  Summing and using
$|I\times D_\rho|\asymp_\Omega\delta\rho^{d+2}$ gives
\eqref{eq:shrinking-measure-estimate}.
\end{proof}

\begin{proposition}[One-step oscillation reduction]\label{prop:root-oscillation-reduction}
There are structural constants $\vartheta_*,\eta_*\in(0,1)$ and $\rho_*>0$, with
$4\rho_*\le\rho_\Omega$, such that for every selected root $w\in\{u,v\}$ and every interior
or boundary cylinder $Q_{4\rho}^\Omega(t_0,x_0)$ contained in a positive-time strip, with
$0<\rho\le\rho_*$,
\begin{equation}\label{eq:one-step-oscillation-reduction}
 \esssup_{Q_{\vartheta_*\rho}^\Omega}w
 -\essinf_{Q_{\vartheta_*\rho}^\Omega}w
 \le(1-\eta_*)
 \left(
 \esssup_{Q_{4\rho}^\Omega}w
 -\essinf_{Q_{4\rho}^\Omega}w
 \right).
\end{equation}
The constants are independent of the distance of $t_0$ to a finite upper time face.
\end{proposition}

\begin{proof}
Let $\mu^+,\mu^-$ be the essential supremum and infimum of $w$ on $Q_{4\rho}^\Omega$ and put
$\omega=\mu^+-\mu^-$.  If $\omega=0$, the conclusion is immediate; hence assume
$\omega>0$.  On the positive-time
cylinder the strong representative is continuous; hence the slice values used below lie between
$\mu^-$ and $\mu^+$ (a violation at one point would persist on a spacetime neighborhood of
positive measure).  Set $m=(\mu^++\mu^-)/2$ and
$s=t_0-\delta_0\rho^2$.  At time $s$, one of
$\{w(s)\le m\}\cap D_\rho$ and $\{w(s)\ge m\}\cap D_\rho$ has measure at least
$|D_\rho|/2$.  Treat the first alternative.  Lemma~\ref{lem:root-measure-propagation}, with
$k=m$ and $H=\omega/2$, gives
\begin{equation}\label{eq:phase-persistence-oscillation-proof}
 |\{w(t)\le\mu^+-M\}\cap D_{\lambda_0\rho}|
 \ge\gamma_0|D_{\lambda_0\rho}|,
 \qquad
 M:=\frac{\epsilon_0\omega}{2},
\end{equation}
for every $t\in(s,t_0)$.

Apply Lemma~\ref{lem:root-shrinking} on this terminal slab with
$\bar\rho=\lambda_0\rho$, base level $k_0=\mu^+-M$, phase fraction $\gamma_0$, aspect ratio
$\bar\delta=\delta_0/\lambda_0^2$, and buffer factor $\Lambda=\lambda_0^{-1}$.
Choose
\[
 \widehat\rho=\frac14\min\{\lambda_0,\sqrt{\delta_0}\}\rho.
\]
Then $Q_{2\widehat\rho}^\Omega(t_0,x_0)$ lies inside
$(s,t_0)\times D_{\bar\rho}$.  Uniform measure density gives a structural $\chi_*>0$ with
\[
 \frac{|Q_{2\widehat\rho}^\Omega|}{|(s,t_0)\times D_{\bar\rho}|}\ge\chi_*.
\]
The shrinking lemma yields
\[
 \frac{|\{w>\mu^+-M2^{-J}\}\cap((s,t_0)\times D_{\bar\rho})|}
 {|(s,t_0)\times D_{\bar\rho}|}
 \le\frac{C_{\rm sh}}{\sqrt J}.
\]
Choose a structural $J_*$ so large that $C_{\rm sh}/\sqrt{J_*}\le\nu_*\chi_*$.  The critical-mass
lemma applied in $Q_{2\widehat\rho}^\Omega$ with
$k=\mu^+-M2^{-J_*}$ and $H=M2^{-J_*}$ then gives
\[
 \esssup_{Q_{\widehat\rho}^\Omega}w
 \le\mu^+-\frac{M}{2^{J_*+1}}.
\]
Thus the oscillation decreases by the fixed fraction
$\eta_*=\epsilon_0/2^{J_*+2}$.  For the second half-measure alternative apply the same three
lemmas to the oriented variable from \eqref{eq:oriented-ternary-root-equation} with
$\varepsilon=-1$; its upper truncations are exactly lower truncations of the original root, at
levels that still lie in the invariant interval.  Decrease $\rho_*$ if necessary so that every cap
and buffer used above lies below $\rho_\Omega$ while retaining $4\rho_*\le\rho_\Omega$.
Taking $\vartheta_*=\widehat\rho/\rho$ proves
\eqref{eq:one-step-oscillation-reduction}.
\end{proof}

\begin{lemma}[Oscillation decay implies terminal boundary H\"older control]
\label{lem:oscillation-to-terminal-holder-ternary}
Let $w$ be continuous and bounded on $[\tau_0,T)\times\overline\Omega$, where $T\le\infty$.
Assume there are $\vartheta,\eta\in(0,1)$ and $\rho_*>0$ such that every admissible backward
interior or boundary cylinder with $0<\rho\le\rho_*$ satisfies
\begin{equation}\label{eq:abstract-oscillation-contraction-ternary}
 \operatorname*{osc}_{Q_{\vartheta\rho}^\Omega(t_0,x_0)}w
 \le(1-\eta)\operatorname*{osc}_{Q_{4\rho}^\Omega(t_0,x_0)}w.
\end{equation}
Then for every $\tau_1>\tau_0$ there exist $\theta\in(0,1)$ and $C<\infty$, depending only on
the displayed structural data, the uniform local geometry of $\Omega$, the bound for $w$, and
$\tau_1-\tau_0$, such that
\begin{equation}\label{eq:abstract-terminal-holder-ternary}
 |w(t,x)-w(s,y)|
 \le C\bigl(|x-y|+|t-s|^{1/2}\bigr)^\theta
\end{equation}
for all $(t,x),(s,y)\in[\tau_1,T)\times\overline\Omega$.  If $T<\infty$, the constant is
independent of the distances of $t$ and $s$ from the upper time face $T$.
\end{lemma}

\begin{proof}
Set
\[
 q:=\frac{\vartheta}{4}\in(0,1),\qquad
 \sigma:=1-\eta\in(0,1),\qquad
 \theta_0:=\frac{\log(1/\sigma)}{\log(1/q)}>0,
\]
and choose any $\theta\in(0,\min\{\theta_0,1\})$.  Iterating
\eqref{eq:abstract-oscillation-contraction-ternary} with
$\rho_{n+1}=q\rho_n$ gives, whenever the largest backward cylinder is admissible,
\[
 \operatorname*{osc}_{Q_{4\rho_n}^\Omega(t_0,x_0)}w
 \le \sigma^n\operatorname*{osc}_{Q_{4\rho_0}^\Omega(t_0,x_0)}w
 \le C q^{n\theta}\|w\|_{L^\infty}.
\]
Hence for every $0<r\le4\rho_0$,
\begin{equation}\label{eq:campanato-oscillation-ternary}
 \operatorname*{osc}_{Q_r^\Omega(t_0,x_0)}w
 \le C\|w\|_{L^\infty}\left(\frac r{\rho_0}\right)^\theta.
\end{equation}

Fix $\tau_1>\tau_0$ and choose $\rho_0>0$ uniformly so small that
$4\rho_0\le\rho_*$, all corresponding caps lie in the uniform boundary-geometry regime, and
$16\rho_0^2<\tau_1-\tau_0$.  Thus every backward cylinder of radius at most $4\rho_0$ whose
upper time lies in $[\tau_1,T)$ stays above the lower time face $\tau_0$.  Let
$X=(t,x)$ and $Y=(s,y)$ lie in $[\tau_1,T)\times\overline\Omega$, assume $t\ge s$, and set
$d_p(X,Y)=|x-y|+|t-s|^{1/2}$.  If $d_p(X,Y)\ge\rho_0/4$, boundedness gives
\eqref{eq:abstract-terminal-holder-ternary} after enlarging $C$.  Otherwise choose a backward
cylinder with upper time $t$, center $x$, and radius $r=4d_p(X,Y)$ (or a fixed comparable radius
when $d_p=0$).  Both points belong to the closure of this cap-cylinder, and continuity together
with \eqref{eq:campanato-oscillation-ternary} yields
\[
 |w(X)-w(Y)|\le C d_p(X,Y)^\theta.
\]
Only backward cylinders occur, so the construction never uses data beyond the later of the two
times.  This proves the terminal-uniform assertion.
\end{proof}

\begin{theorem}[Spectral De Giorgi regularity in arbitrary dimension]\label{thm:spectral-DG-holder}
Let $d\ge1$, let $\Omega\subset\R^d$ be bounded and connected with $C^{2+\alpha_0}$ boundary
for some $0<\alpha_0<1$, as in Theorem~\ref{thm:main}, and let the three pair frictions be positive and pairwise distinct, ordered as
$0<\alpha<\beta<\gamma$.  Let $y$ be a maximal strong solution and suppose all three conserved
species masses are positive.  Then for every $\tau\in(0,T_{\max})$ there exist
$\theta\in(0,1)$ and $C_\tau<\infty$ such that
\begin{equation}\label{eq:spectral-holder-terminal}
 \|u\|_{C^{\theta/2,\theta}([\tau,T_{\max})\times\overline\Omega)}
 +\|v\|_{C^{\theta/2,\theta}([\tau,T_{\max})\times\overline\Omega)}
 \le C_\tau.
\end{equation}
No smallness assumption is used.
\end{theorem}

\begin{proof}
Apply Lemma~\ref{lem:forward-spectral-rectangle} at time $\tau/2$.  It gives one compact separated
rectangle containing both root ranges on
$[\tau/2,T_{\max})\times\overline\Omega$.  Proposition~\ref{prop:root-oscillation-reduction}
therefore supplies the same structural one-step contraction on every sufficiently small backward
interior or boundary cylinder in this strip.  Apply
Lemma~\ref{lem:oscillation-to-terminal-holder-ternary} with $\tau_0=\tau/2$ and
$\tau_1=\tau$, first to $u$ and then to $v$.  The invariant rectangle supplies the required uniform
$L^\infty$ bounds.  This gives \eqref{eq:spectral-holder-terminal}; because the abstract lemma uses
only backward cylinders, the constant is uniform as $t\uparrow T_{\max}$ if the maximal time is
finite.
\end{proof}

\begin{remark}[Dependence and order of the constants]\label{rem:DG-constant-ledger}
The constants in the Caccioppoli, logarithmic propagation, shrinking, critical-mass, and
oscillation estimates depend only on the fixed separated spectral rectangle, $d$, and the uniform
local geometry of $\Omega$.  The constants may be chosen in the following order:
\[
\begin{gathered}
\mathcal R,\Omega
\longrightarrow
(\varepsilon_{\mathcal R},\rho_\Omega)
\longrightarrow
\lambda_0
\longrightarrow
A_0
\longrightarrow
\epsilon_0
\longrightarrow
\delta_0,\gamma_0
\longrightarrow
\nu_*\\
\longrightarrow
(\bar\delta,\Lambda,\chi_*,C_{\rm sh})
\longrightarrow
J_*
\longrightarrow
(\vartheta_*,\eta_*).
\end{gathered}
\]
Here $\varepsilon_{\mathcal R}$ is the entropy-stabilization parameter,
$\lambda_0$ is fixed solely by the thin-annulus modulus, $A_0$ is then structural,
$\epsilon_0$ is chosen from the logarithmic ratio, and only afterward is the propagation time
$\delta_0$ fixed.  The critical-mass number $\nu_*$ comes from the quadratic Caccioppoli/Sobolev
recursion and is independent of those logarithmic choices.  Finally $J_*$ is selected only after the
fixed slab-to-cylinder density $\chi_*$ and shrinking constant $C_{\rm sh}$ are known.  Each choice therefore depends only on quantities fixed earlier in the displayed order.  In particular all constants are
independent of the truncation height $H$, the oscillation size, the level, the time location of the
cylinder, and any a priori terminal $C^1$ norm.  By Lemma~\ref{lem:C11-entropy-chain}, all scalar
and entropy testing identities are justified directly in the maximal-$L^p$ strong class; no
terminal-uniform classical Schauder estimate enters the De Giorgi argument.  Qualitative
positive-time classicality from the local theory is used in the cited positivity theorem; its
constants on compact preterminal intervals need not be uniform near $T_{\max}$.  This is why
Theorem~\ref{thm:spectral-DG-holder} retains the boundary hypothesis of the main theorem, although
the local geometric estimates themselves require only $C^2$ boundary.
\end{remark}

\begin{remark}[Required size of the De Giorgi exponent]
No lower bound on $\theta$ relative to $d+2$ or to a maximal-$L^p$ trace exponent is required.  The only property of $\theta$ used below is that the state-space oscillation becomes arbitrarily
small on sufficiently small cylinders.  Thus the De Giorgi step imposes no upper restriction on the spatial dimension.
\end{remark}

\section{Local symmetrization by the mixing entropy}\label{sec:symmetrizer}

By \eqref{eq:mixing-H-positive}, $H(V)=D_V^2\mathfrak h(V)$ is positive definite and defines the
local entropy metric used below.  Let
\[
\mathcal K=\{(u+v,uv):u_-\le u\le u_+<v_-\le v\le v_+\}
\]
be the compact trace--determinant image of a separated spectral rectangle.  For $V_*\in\mathcal K$
set
\[
H_*=H(V_*)=D_V^2\mathfrak h(V_*),\qquad P_*=H_*^{1/2},
\]
and
\[
B_*(V)=P_*A(V)P_*^{-1}.
\]
Then
\begin{equation}\label{eq:frozen-H-identity}
\sym B_*(V)
=P_*^{-1}\sym\bigl(H_*A(V)\bigr)P_*^{-1}.
\end{equation}

\begin{lemma}[Uniform local strong parabolicity in the frozen entropy metric]
\label{lem:local-frozen-symmetrizer}
There exist $\delta_*>0$, $\lambda_*>0$, and $C_{\rm sym}<\infty$, depending only on $\mathcal K$, such
that for every $V_*\in\mathcal K$,
\begin{equation}\label{eq:local-strong}
|V-V_*|\le\delta_*
\quad\Longrightarrow\quad
\sym B_*(V)\ge\lambda_*I,
\end{equation}
and
\[
\|P_*\|+\|P_*^{-1}\|\le C_{\rm sym}.
\]
\end{lemma}

\begin{proof}
By \eqref{eq:mixing-H-positive}, $H(V)$ and $H(V)A(V)$ are continuous positive-definite matrices
on the compact set $\mathcal K$.  Therefore there are constants $m_H,M_H,m_D>0$ such that
\[
m_HI\le H(V)\le M_HI,
\qquad
H(V)A(V)\ge m_DI
\]
for every $V\in\mathcal K$.  Let $L_A$ be a Lipschitz constant for $A$ on one fixed neighborhood
of $\mathcal K$.  Then
\[
\|P_*\|\le M_H^{1/2},\qquad \|P_*^{-1}\|\le m_H^{-1/2}.
\]
Moreover
\begin{align*}
\sym(H_*A(V))
&=H_*A(V_*)+\sym\bigl(H_*(A(V)-A(V_*))\bigr)\\
&\ge \bigl(m_D-M_HL_A|V-V_*|\bigr)I.
\end{align*}
Fix an open state neighborhood $\mathcal N$ of the compact range $\mathcal K$
on which $A$ is Lipschitz with constant $L_A$, and choose $r_A>0$ so that the closed
$r_A$-neighborhood of $\mathcal K$ is contained in $\mathcal N$.  Thus, when $L_A>0$, one may take
explicitly
\begin{equation}\label{eq:quantitative-symmetrizer-radius}
\delta_*:=\min\left\{r_A,\frac{m_D}{2M_HL_A}\right\},
\qquad
\lambda_*:=\frac{m_D}{2M_H}.
\end{equation}
If $L_A=0$, take $\delta_*=r_A$ and the same value of $\lambda_*$ satisfies the estimate.  In particular every state
$V$ with $V_*\in\mathcal K$ and $|V-V_*|\le\delta_*$ lies in the neighborhood on which the
Lipschitz estimate is valid.  For such $V$ and every target vector $\xi$,
\[
\xi^{\sf T}\sym B_*(V)\xi
=(P_*^{-1}\xi)^{\sf T}\sym(H_*A(V))(P_*^{-1}\xi)
\ge \frac{m_D}{2}|P_*^{-1}\xi|^2
\ge \lambda_*|\xi|^2.
\]
All constants are uniform in the center state.
\end{proof}

\section{Conormal Schauder regularity}\label{sec:first-schauder}

We next convert local state-space symmetrizability into uniform interior and boundary $C^{1+a}$ estimates.  After flattening the physical boundary, the local conormal Schauder theory of Dong--Jeon applies directly to the homogeneous system on a half-cylinder.  Thus no artificial boundary, cutoff equation, or interpolation absorption is needed.

\begin{lemma}[Localized weak conormal regularity]\label{lem:localized-conormal-schauder-ternary}
Let $0<\beta<1$ and
\[
 \mathcal Q_R^+:=(-R^2,0)\times(B_R'\times(0,R)).
\]
Suppose $q:\mathcal Q_4^+\to\mathbb R^2$ is a bounded weak solution of
\begin{equation}\label{eq:abstract-conormal-system-ternary}
 q_t-\partial_\alpha\bigl(\mathbb A^{\alpha\beta}(t,x)\partial_\beta q\bigr)=0
\end{equation}
with homogeneous conormal condition on the flat face.  Assume that the coefficient tensor is
uniformly bounded and strongly parabolic on real arrays and that
\[
 \mathbb A\in C^{\beta/2,\beta}_{t,x}(\overline{\mathcal Q_4^+})
\]
with fixed H\"older norm.  Then
\begin{equation}\label{eq:localized-conormal-lipschitz-ternary}
 \|\nabla q\|_{L^\infty(\mathcal Q_2^+)}
 \le C\|q\|_{L^\infty(\mathcal Q_4^+)},
\end{equation}
and, for every $0<a<\beta$,
\begin{equation}\label{eq:localized-conormal-C1a-ternary}
 \|q\|_{C^{(1+a)/2,1+a}_{t,x}(\overline{\mathcal Q_2^+})}
 \le C_a\|q\|_{L^\infty(\mathcal Q_4^+)}.
\end{equation}
The analogous estimate holds on full cylinders.  The statement remains valid when the original
spatial dimension is one, by adjoining one dummy tangential variable in the boundary case; full
cylinders may similarly be reduced to the half-cylinder statement by adjoining one dummy normal
variable.
\end{lemma}

\begin{proof}
Set the weight exponent in Dong--Jeon~\cite{DongJeon2026} equal to zero and the divergence datum
to zero.  Their coefficient condition (1.2) then reduces to boundedness and coercivity on real
arrays, exactly the strong-parabolicity hypothesis above, and their weak formulation encodes the
homogeneous conormal boundary condition.  Since a $C^{\beta/2,\beta}$ coefficient has Dini mean
oscillation, Theorem~3.1 together with Lemma~3.5 yields the quantitative spatial Lipschitz estimate
on $\mathcal Q_2^+$ from the norm on $\mathcal Q_4^+$.
For $0<a<\beta$ the same coefficient field belongs to $C^{a/2,a}$, so Theorem~4.1 with $k=1$
and H\"older exponent $a$, again with zero divergence datum, gives
$C^{(1+a)/2,1+a}$ regularity on $\mathcal Q_{1/2}^+$ from $\mathcal Q_1^+$.
Parabolic rescaling therefore yields the stated estimate from $\mathcal Q_4^+$ to $\mathcal Q_2^+$.
The constants in \eqref{eq:localized-conormal-lipschitz-ternary}--\eqref{eq:localized-conormal-C1a-ternary}
depend only on the displayed ellipticity, boundedness, and coefficient-modulus data.
Dong--Jeon's parabolic theorem is stated for spatial dimension at least two; the dummy-variable
reductions in the next paragraph cover the one-dimensional case.

For a full cylinder, adjoin one normal variable, extend $q$ and the original coefficients
constantly in that variable, put an identity block in the new normal--normal coefficient and zero
mixed blocks, and apply the half-cylinder result to the augmented conormal system.  For a
one-dimensional physical boundary point, adjoin instead one dummy tangential variable.  Restriction
to a constant dummy-variable slice gives the asserted estimates.
\end{proof}

\begin{proposition}[Uniform terminal gradient H\"older gain]\label{prop:first-Schauder-gain}
Assume the hypotheses of Theorem~\ref{thm:spectral-DG-holder} and suppose $T_{\max}<\infty$.  Let
\[
0<a<\min\{\theta,\alpha_0\}.
\]
Then, after moving to a sufficiently late time $\tau_2<T_{\max}$,
\begin{equation}\label{eq:first-Schauder}
V\in C^{(1+a)/2,1+a}
([\tau_2,T_{\max})\times\overline\Omega)
\end{equation}
with a norm finite and uniform up to $T_{\max}$.
\end{proposition}

\begin{proof}
Fix a positive-time strip on which Theorem~\ref{thm:spectral-DG-holder} and the separated spectral rectangle both hold.  On every compact substrip the maximal strong solution belongs to the weak conormal class required below because $p>d+2>2$ gives local $L^2$ control of the solution and its spatial gradient.  Thus the estimates can be applied directly at the strong-solution level, and their constants are independent of the preterminal upper endpoint.

Theorem~\ref{thm:spectral-DG-holder} gives one uniform terminal H\"older modulus
\[
|V(z)-V(z_0)|
\le C_H d_{\rm par}(z,z_0)^\theta,
\qquad
d_{\rm par}\bigl((t,x),(s,y)\bigr):=|x-y|+|t-s|^{1/2}.
\]
Write $\tau_0$ for the lower end of this strip.  Fix finitely many boundary graph charts and
let $C_\Phi\ge2$ dominate their spatial distortion and the corresponding inverse distortion,
including the fixed cylinder-containment factors.  Choose one radius $r_*>0$ so small that
\begin{equation}\label{eq:fixed-Sch-radius}
8C_\Phi r_*<r_{\rm chart},\qquad
64r_*^2<T_{\max}-\tau_0,\qquad
C_H(8C_\Phi r_*)^\theta\le \frac{\delta_*}{2}.
\end{equation}
Use the terminal strip starting at $\tau_2:=\tau_0+64r_*^2<T_{\max}$.  Then every backward
cylinder and enlarged chart cylinder used below lies entirely inside the separated H\"older strip.

On a backward cylinder centered at $z_0=(t_0,x_0)$ freeze
\[
P_*=H(V(z_0))^{1/2},\qquad q=P_*(V-V(z_0)).
\]
Then
\begin{equation}\label{eq:q-local-system}
q_t-\Div(B_*(V)\nabla q)=0,
\qquad
\sym B_*(V)\ge\lambda_*I,
\qquad
B_*(V)\in C^{\theta/2,\theta},
\end{equation}
with constants uniform in the center and in the preterminal upper endpoint.  By \eqref{eq:fixed-Sch-radius}, the whole $4r_*$ cylinder stays inside the frozen state ball on which the lower bound in \eqref{eq:q-local-system} holds.

Consider first a physical boundary point.  After a rigid motion, choose the graph chart so that
$\Omega$ is locally $\{x_d>\psi(x')\}$ and
\[
x=\Phi(y)=\bigl(y',y_d+\psi(y')\bigr),\qquad y_d>0.
\]
Because the boundary atlas is finite, $r_*$ in \eqref{eq:fixed-Sch-radius} may also be chosen so
that, for every chart center, the image of a fixed Euclidean half-ball of radius $8r_*$ is contained
in the chart neighborhood and its inverse image contains all half-balls used below.  Thus the Dong--Jeon estimate is applied on Euclidean half-cylinders contained in the flattened
chart.
The Jacobian $J=D\Phi$ is triangular and
\begin{equation}\label{eq:det-volume-chart}
\det J\equiv1.
\end{equation}
Let $\widetilde q=q\circ\Phi$ and $\widetilde B_*=B_*(V)\circ\Phi$.  The Piola identity gives, for
any spatial flux $\mathcal F$,
\[
 (\Div_x\mathcal F)\circ\Phi
 =\Div_y(J^{-1}\widetilde{\mathcal F})
\]
because $\det J=1$, while $\nabla_xq\circ\Phi=J^{-T}\nabla_y\widetilde q$.  Since the spatial
Jacobian acts only on spatial indices and $\widetilde B_*$ only on the two target components, the
flattened equation is
\[
 \widetilde q_t-\partial_\alpha
 \bigl(\mathbb A^{\alpha\beta}\partial_\beta\widetilde q\bigr)=0,
 \qquad
 \mathbb A^{\alpha\beta}(t,y)
 =G^{\alpha\beta}(y)\widetilde B_*(t,y),
 \qquad
 G:=J^{-1}J^{-T}.
\]
Equivalently,
\begin{equation}\label{eq:flattened-tensor}
\mathbb A^{\alpha\beta}(t,y)
=G^{\alpha\beta}(y)\,\widetilde B_*(t,y).
\end{equation}
On $y_d=0$, $J^{-T}e_d=(-\nabla\psi,1)$ is a nonzero multiple of the physical normal.  The
original homogeneous Neumann condition implies $B_*(V)\partial_\nu q=0$, and therefore the
$d$th component of the transformed flux vanishes.  Hence the flat boundary condition is exactly
\begin{equation}\label{eq:flat-conormal}
\sum_\beta G^{d\beta}(y)\widetilde B_*(t,y)\,\partial_\beta\widetilde q=0
\qquad\text{on }y_d=0.
\end{equation}
For arbitrary real target vectors $\xi_\alpha\in\mathbb R^2$, symmetry of $G$ makes the target skew part cancel, and therefore
\begin{align}
\sum_{\alpha,\beta}G^{\alpha\beta}
 \xi_\alpha^{T}\widetilde B_*\xi_\beta
&=\sum_{\alpha,\beta}G^{\alpha\beta}
 \xi_\alpha^{T}\sym(\widetilde B_*)\xi_\beta\\
&\ge \lambda_{\min}(G)\lambda_*\sum_\alpha|\xi_\alpha|^2.
\label{eq:flattened-strong-ellipticity}
\end{align}
Thus $\mathbb A$ satisfies the uniform parabolicity condition used by Dong--Jeon.  The corresponding
upper bound in their condition (1.2) is uniform as well, because the finite chart atlas bounds
$G$ and the compact invariant state range bounds $B_*$.  Moreover, choose an intermediate exponent $\beta_{\rm Sch}$ with
\[
 a<\beta_{\rm Sch}<\min\{\theta,\alpha_0\}.
\]
The metric $G$ is spatially more regular than $C^{\beta_{\rm Sch}}$, and
\eqref{eq:q-local-system} implies
\[
 \mathbb A\in C^{\beta_{\rm Sch}/2,\beta_{\rm Sch}}
\]
with one structural norm bound on every rescaled boundary cylinder.  Thus the strict
exponent inequality required by Lemma~\ref{lem:localized-conormal-schauder-ternary} is satisfied.

The transformed solution is a bounded weak conormal solution of the homogeneous system: the
maximal-$L^p$ strong class supplies the local $L^2$ regularity needed in the weak formulation, the
Piola transform gives the divergence identity, and \eqref{eq:flat-conormal} is the physical
no-flux condition in flat coordinates.  Therefore
Lemma~\ref{lem:localized-conormal-schauder-ternary}, applied after parabolic rescaling with
$\beta=\beta_{\rm Sch}>a$, gives
\begin{equation}\label{eq:DJ-local-Sch}
\|q\|_{C^{(1+a)/2,1+a}(Q_{2r_*}^+)}
\le C_{\rm DJ}\|q\|_{L^\infty(Q_{4r_*}^+)}.
\end{equation}
The constant may depend on the fixed radius, dimension, $a$, the uniform parabolicity and
boundedness constants, and the common $C^{a/2,a}$ coefficient bound, but not on the cylinder
center or the preterminal upper endpoint.

The state range is compact and both $P_*$ and $P_*^{-1}$ are uniformly bounded by Lemma~\ref{lem:local-frozen-symmetrizer}.  Since $P_*$ is constant on each frozen cylinder, $\nabla V=P_*^{-1}\nabla q$ there, so \eqref{eq:DJ-local-Sch} gives the corresponding estimate for $V$ with the same uniformity.  Also
\[
\|q\|_{L^\infty(Q_{4r_*})}
\le C\,\operatorname*{osc}_{Q_{4r_*}}V
\le C
\]
uniformly by the compact invariant state range.  A finite spatial covering and sliding backward cylinders therefore give a uniform $C^{(1+a)/2,1+a}$ estimate on a late terminal strip.  Pairs of points whose parabolic distance is larger than a fixed fraction of $r_*$ are controlled directly by boundedness and the local gradient bound, so the local estimates patch to the global strip norm.  No estimate uses data on the upper time face.  Letting $T\uparrow T_{\max}$ proves \eqref{eq:first-Schauder}.
\end{proof}

\begin{remark}[Conormal Schauder input]\label{rem:DJ-input}
The boundary regularity input is the published Dong--Jeon theorem for conormal parabolic systems on a flat half-cylinder.  Setting their degeneracy/singularity exponent equal to zero recovers the uniformly parabolic case used here.  For $k=1$, Theorem~4.1 assumes $C^{a/2,a}$ leading coefficients and data and yields $C^{(1+a)/2,1+a}$ regularity; the case is obtained in their paper from the quantitative Lipschitz and space--time modulus estimates in Section~3.  Because our localized equation is homogeneous, there are no cutoff-generated lower-order data and no artificial boundary or initial compatibility condition to verify.
\end{remark}
\section{Terminal trace control by frozen maximal regularity}\label{sec:terminal-trace}

The conormal estimate gives a uniform $C^{1+a}$ bound on a late strip and hence a uniform spatial
gradient bound.  The gradient bound controls the quadratic forcing obtained by expanding the reduced
divergence operator, while the full $C^{1+a}$ bound implies relative compactness of the frozen
principal coefficient family in $C^1$.  The temporal freezing error requires only the earlier De
Giorgi H\"older modulus of the state.  Thus these estimates enter distinct parts of the terminal
argument, and the full quasilinear HMPW substitution operator need not be frozen.

Set
\[
 Y_0:=L_p(\Omega;\mathbb R^2),\qquad
 Y_1:=\{w\in W_p^2(\Omega;\mathbb R^2):\partial_\nu w=0\},
\]
and
\[
 Y_{\gamma,p}:=(Y_0,Y_1)_{1-1/p,p}
 =\{w\in W_p^{2-2/p}(\Omega;\mathbb R^2):\partial_\nu w=0\}.
\]
Since $y_3=1-y_1-y_2$, a bound for $y'$ in $Y_{\gamma,p}$ is equivalent, up to fixed constants and
a constant shift, to a bound for $y-\mathbf e/3$ in $X_{\gamma,p}$.

Write
\[
 D(\eta):=C(\eta)^{-1}.
\]
On a neighborhood of the closed physical reduced simplex, $D$ is smooth because
$\det C=\Pi$ has a positive lower bound.  Expanding \eqref{eq:reduced-system} in nondivergence
form gives, in the maximal-regularity class,
\begin{equation}\label{eq:reduced-nondivergence-terminal}
 y'_t-D(y')\Delta y'=F(y',\nabla y'),
\end{equation}
where
\begin{equation}\label{eq:terminal-quadratic-forcing}
 F(y',\nabla y')
 :=\sum_{k=1}^d DD(y')[\partial_k y']\,\partial_k y'.
\end{equation}

\begin{proposition}[Uniform terminal trace bound]\label{prop:terminal-trace-bound}
Assume that the three frictions are pairwise distinct, all three conserved species masses are
positive, and $T_{\max}<\infty$.  Then there is a late time $\tau_3<T_{\max}$ such that
\begin{equation}\label{eq:terminal-Xgamma-bound}
\sup_{\tau_3\le t<T_{\max}}\|y(t)-\mathbf e/3\|_{X_{\gamma,p}}<\infty.
\end{equation}
\end{proposition}

\begin{proof}
Choose $0<a<\min\{\theta,\alpha_0\}$ as in
Proposition~\ref{prop:first-Schauder-gain}.  Since the maps $y'\mapsto V$ and $V\mapsto y'$ are
affine with invertible constant Jacobian, that proposition gives, after moving to a late strip,
\begin{equation}\label{eq:terminal-C1a-reduced}
 \sup_{\tau_3\le t<T_{\max}}\|y'(t)\|_{C^{1+a}(\overline\Omega)}\le M_1.
\end{equation}
On the other hand Theorem~\ref{thm:spectral-DG-holder}, together with the smooth root-to-$V$ map
on the invariant rectangle and the affine $V$-to-$y'$ map, gives
\begin{equation}\label{eq:terminal-state-time-holder}
 \|y'(t)-y'(s)\|_{C(\overline\Omega)}
 \le M_0|t-s|^{\theta/2},
 \qquad \tau_3\le s,t<T_{\max}.
\end{equation}
No time H\"older estimate for $\nabla y'$ is used below.

For $t\in[\tau_3,T_{\max})$ define the frozen principal operator
\begin{equation}\label{eq:terminal-frozen-principal-operator}
 A(t)w:=-D(y'(t,\cdot))\Delta w,
 \qquad D(A(t))=Y_1.
\end{equation}
The bound \eqref{eq:terminal-C1a-reduced} implies relative compactness of the state orbit in
$C^1(\overline\Omega;\mathbb R^2)$.  Consequently,
\[
 \mathfrak A:=\overline{\{A(t):\tau_3\le t<T_{\max}\}}
\]
is compact in $\mathcal L(Y_1,Y_0)$.  Every operator in this closure has maximal $L^p$ regularity.
Indeed, if $A(t_n)\to A_*$, then after passing to a subsequence
$y'(t_n)\to\eta_*$ in $C^1$.  The limit remains physical pointwise, and
\[
 A_*w=-D(\eta_*)\Delta w.
\]
Under the fixed two-component elimination map, $D(\eta_*)$ is precisely the reduced
coordinate representation of the normally elliptic HMPW principal diffusion matrix $A_0$.
Consequently its spectrum lies in $(0,\infty)$ uniformly on the compact physical state set.
The Neumann Lopatinskii--Shapiro condition can also be verified directly.  Freeze at a boundary point, write $D_*=D(\eta_*(x_0))$, let
$\lambda\in\overline{\mathbb C_+}$ and $\xi\in\mathbb R^{d-1}$ with
$|\lambda|+|\xi|\ne0$, and consider a decaying boundary-symbol solution.  Its normal profile
$\phi$ satisfies
\[
 \phi''=M_{\lambda,\xi}\phi,
 \qquad
 M_{\lambda,\xi}:=|\xi|^2I+\lambda D_*^{-1},
 \qquad
 \phi'(0)=0.
\]
The spectrum of $D_*^{-1}$ is positive real, so $M_{\lambda,\xi}$ is invertible and its spectrum
avoids the negative real axis.  Its principal square root therefore has spectrum in the open right
half-plane.  Equivalently, for the first-order system associated with $\phi''=M_{\lambda,\xi}\phi$,
the stable spectral subspace is the graph of $-M_{\lambda,\xi}^{1/2}$.  Hence every decaying
solution is of the form $\phi(r)=e^{-rM_{\lambda,\xi}^{1/2}}c$; this functional-calculus statement
does not require $D_*$ to be diagonalizable.  The Neumann condition gives
$M_{\lambda,\xi}^{1/2}c=0$, hence $c=0$.  For prescribed Neumann data
$\phi'(0)=b\in\mathbb C^2$, the unique decaying profile is
\[
 \phi(r)=-e^{-rM_{\lambda,\xi}^{1/2}}M_{\lambda,\xi}^{-1/2}b.
\]
Thus the boundary-symbol map is onto as well as one-to-one.  This proves the Neumann
complementing condition; compare HMPW Section~3.1.  Since the coefficient field is $C^1$ and
$\partial\Omega$ is $C^2$, the parameter-elliptic boundary theorem of
Denk--Hieber--Pr\"uss~\cite[Theorem~8.2]{DenkHieberPruess2003}, also invoked in
\cite[Section~3.1]{HerbergMeyriesPruessWilke2017}, yields maximal $L^p$ regularity for $A_*$ on finite time intervals with
the common domain $Y_1$.  If the parameter-elliptic theorem is formulated after a spectral
shift, that shift is removed on a fixed finite interval by an exponential change of the unknown.
An explicit finite-interval formulation covering the same realization is
\cite[Theorem~6.2 in the preprint version, with $q=p$]{Augner2024}.  After complexification,
take its $E=\mathbb C^2$, $m=m_1=1$, $P_{1,1}=I$, $P_{1,0}=0$, and zero boundary data.  The principal
coefficients are continuous, the lower-order interior coefficients vanish, and the Neumann
coefficients are time-independent and $C^1$ on the boundary.  For its boundary exponent
$\kappa=\tfrac12-\tfrac1{2p}$, the coefficient condition is met with both integrability indices
equal to $p$, since
\[
 \frac1p+\frac{d-1}{2p}<\kappa
 \quad\Longleftrightarrow\quad p>d+2.
\]
The theorem's initial space is $B_{p,p}^{2-2/p}=W_p^{2-2/p}$.  Since $p>d+2\ge3$ gives
$\kappa>1/p$, its initial--boundary compatibility condition is exactly
$\partial_\nu w(0)=0$.  Thus it yields the same domain $Y_1$ and trace space $Y_{\gamma,p}$,
without an additional boundary constraint.

We make the uniformity over $\mathfrak A$ explicit, including on short intervals.  Fix a reference
length $h_*>0$ and set
\[
 \mathbb E_1^0(h_*):=\{w\in W^{1,p}(0,h_*;Y_0)\cap L^p(0,h_*;Y_1):w(0)=0\},
 \qquad
 \mathbb E_0(h_*):=L^p(0,h_*;Y_0).
\]
For $A\in\mathfrak A$, the map
\[
 \mathcal T_Aw:=w_t+Aw
\]
is an isomorphism from $\mathbb E_1^0(h_*)$ to $\mathbb E_0(h_*)$.  Maximal regularity is stable
under small perturbations in $\mathcal L(Y_1,Y_0)$; a finite Neumann-series covering of the compact
set $\mathfrak A$ therefore gives one uniform zero-trace inverse bound.  Compactness also gives a
uniform bound for $\|A\|_{\mathcal L(Y_1,Y_0)}$.

Fix one bounded coretraction of the trace into
$W^{1,p}(0,h_*;Y_0)\cap L^p(0,h_*;Y_1)$.  Subtracting this coretraction reduces nonzero traces to
the preceding zero-trace problem.  On the fixed reference interval, maximal regularity together
with the standard trace embedding therefore gives constants, uniform for $A\in\mathfrak A$, such
that
\[
 \|w\|_{W^{1,p}(0,h_*;Y_0)\cap L^p(0,h_*;Y_1)}
 +\sup_{0\le t\le h_*}\|w(t)\|_{Y_{\gamma,p}}
 \le C_{\rm ref}\bigl(\|w(0)\|_{Y_{\gamma,p}}+\|f\|_{L^p(0,h_*;Y_0)}\bigr).
\]
For an interval of length $0<h<h_*$, extend $f$ by zero to $(0,h_*)$, solve the autonomous
frozen problem on the fixed reference interval with the same initial trace, and restrict.
Uniqueness identifies the restriction with the original solution.  After translation in time,
there are therefore constants $C_{\rm MR},C_{\rm tr}$, independent of $A$, $s$, and $h\le h_*$,
such that for every $I=(s,s+h)$,
\begin{align}
 \|w\|_{W^{1,p}(I;Y_0)\cap L^p(I;Y_1)}
 &\le C_{\rm MR}\bigl(\|w(s)\|_{Y_{\gamma,p}}+\|f\|_{L^p(I;Y_0)}\bigr),
 \label{eq:uniform-frozen-principal-MR}\\
 \sup_{t\in I}\|w(t)\|_{Y_{\gamma,p}}
 &\le C_{\rm tr}\bigl(\|w(s)\|_{Y_{\gamma,p}}+\|f\|_{L^p(I;Y_0)}\bigr).
 \label{eq:uniform-frozen-principal-trace}
\end{align}
Thus no trace embedding on a shrinking interval is used.

The temporal operator modulus now uses only the state H\"older estimate.  Smoothness of $D$ on
the compact physical state set and \eqref{eq:terminal-state-time-holder} give
\begin{equation}\label{eq:terminal-principal-operator-holder}
 \|A(t)-A(s)\|_{\mathcal L(Y_1,Y_0)}
 \le L_D\|y'(t)-y'(s)\|_{C(\overline\Omega)}
 \le L_DM_0|t-s|^{\theta/2}.
\end{equation}
The spatial gradient bound contained in \eqref{eq:terminal-C1a-reduced} and the smoothness of $D$
give, through \eqref{eq:terminal-quadratic-forcing},
\begin{equation}\label{eq:terminal-forcing-uniform}
 \sup_{\tau_3\le t<T_{\max}}
 \|F(y'(t),\nabla y'(t))\|_{Y_0}\le C_F<\infty.
\end{equation}

Let $\tau_3\le s<s+h<T_{\max}$ with $h\le h_*$, and set $I=(s,s+h)$.  Freezing
\eqref{eq:reduced-nondivergence-terminal} at the left endpoint gives
\begin{equation}\label{eq:terminal-frozen-equation-principal}
 y'_t+A(s)y'
 =F(t)+(A(s)-A(t))y'.
\end{equation}
Apply \eqref{eq:uniform-frozen-principal-MR}.  By
\eqref{eq:terminal-principal-operator-holder},
\[
 \|(A(s)-A(t))y'\|_{L^p(I;Y_0)}
 \le L_DM_0h^{\theta/2}
 \|y'\|_{L^p(I;Y_1)}.
\]
Choose $h_0\le h_*$ so small that
$C_{\rm MR}L_DM_0h_0^{\theta/2}\le1/2$.  Using also
\eqref{eq:terminal-forcing-uniform} and absorbing the last term gives
\begin{equation}\label{eq:terminal-principal-MR-short}
 \|y'\|_{W^{1,p}(I;Y_0)\cap L^p(I;Y_1)}
 \le 2C_{\rm MR}\bigl(
 \|y'(s)\|_{Y_{\gamma,p}}+C_Fh^{1/p}\bigr)
\end{equation}
for every $h\le h_0$.  Using \eqref{eq:uniform-frozen-principal-trace} for
\eqref{eq:terminal-frozen-equation-principal}, and estimating its forcing by
\eqref{eq:terminal-forcing-uniform}, \eqref{eq:terminal-principal-operator-holder}, and
\eqref{eq:terminal-principal-MR-short}, yields constants $C_{\rm it}\ge1$ and $C_0<\infty$,
independent of $s$ and $h\le h_0$, such that
\begin{equation}\label{eq:terminal-Xgamma-recursion}
 \sup_{t\in[s,s+h]}\|y'(t)\|_{Y_{\gamma,p}}
 \le C_{\rm it}\|y'(s)\|_{Y_{\gamma,p}}+C_0.
\end{equation}
No contractivity condition $C_{\rm it}<1$ is needed: the terminal strip has finite length.
Indeed, fix $b\in(\tau_3,T_{\max})$ and partition $[\tau_3,b]$ into at most
\[
 N=1+\left\lceil\frac{T_{\max}-\tau_3}{h_0}\right\rceil
\]
consecutive intervals of length at most $h_0$.  Every endpoint is strictly below $T_{\max}$,
so the short-interval estimates and the trace recursion apply on every interval.
Their constants and the number $N$ are independent of $b$.  Since
$y'(\tau_3)\in Y_{\gamma,p}$, finite iteration of the affine recursion
\eqref{eq:terminal-Xgamma-recursion} bounds
$\sup_{\tau_3\le t\le b}\|y'(t)\|_{Y_{\gamma,p}}$ independently of $b$.
Taking the supremum over $b<T_{\max}$ gives
\[
 \sup_{\tau_3\le t<T_{\max}}\|y'(t)\|_{Y_{\gamma,p}}<\infty.
\]
The fixed linear relation $y_3=1-y_1-y_2$ then gives
\eqref{eq:terminal-Xgamma-bound} in the full trace space $X_{\gamma,p}$.
\end{proof}

\begin{remark}[Regularity requirements for continuation]\label{rem:terminal-load-bearing}
The continuation estimate uses the three regularity outputs in different ways.  De Giorgi gives
the temporal H\"older modulus of the state needed in
\eqref{eq:terminal-principal-operator-holder}; the first conormal gain gives the spatial Lipschitz
bound needed for \eqref{eq:terminal-forcing-uniform}; and the stronger $C^{1+a}$ part makes the
frozen principal family compact in the $C^1$ coefficient topology so that the existing
parameter-elliptic maximal-regularity theorem applies uniformly after a finite covering.  No
terminal $C^{2+a}$ estimate and no time H\"older control of $\nabla y$ are required.
\end{remark}

\section{Continuation and global existence}\label{sec:restart}

The terminal trace estimate now allows application of the abstract continuation theorem used in
the Maxwell--Stefan local theory of Herberg--Meyries--Pr\"uss--Wilke.  We verify explicitly the two
continuation hypotheses needed below---boundedness in the unweighted trace space and positive
distance from the boundary of the weighted state set---without using assumptions from their later treatment of reactions.

\begin{proposition}[Continuation from a terminal trace bound]
\label{prop:HMPW-trace-continuation}
Let $y$ be the maximal Maxwell--Stefan diffusion solution from Section~\ref{sec:local}, and
assume all three conserved species masses are positive.  If, for some $\tau<T_{\max}$,
\begin{equation}\label{eq:HMPW-terminal-assumption}
 \sup_{\tau\le t<T_{\max}}
 \|y(t)-\mathbf e/3\|_{X_{\gamma,p}}<\infty,
\end{equation}
then $T_{\max}=\infty$.
\end{proposition}

\begin{proof}
Assume for contradiction that $T_{\max}<\infty$ and put $z=y-\mathbf e/3$.  Choose
\[
 \mu\in(\mu_0,1),\qquad
 \mu_0=\frac{d+2}{2p}+\frac12,
\]
as in the weighted local theory of \cite[Section~3.1]{HerbergMeyriesPruessWilke2017}.  That
section constructs a relative open finite-dimensional state neighborhood $\mathcal V$ containing
the entire closed concentration simplex $\mathcal D$, and the corresponding substitution set
$V_\mu\subset X_{\gamma,\mu}$, where, for the present range of $\mu$,
\[
 X_{\gamma,\mu}
 =\{w\in W_p^{2\mu-2/p}(\Omega;\mathbb E):\partial_\nu w=0\}.
\]  On $V_\mu$ the reduced
Maxwell--Stefan operator $\mathcal A$ is $C^1$, every $\mathcal A(z)$ has maximal $L^p$
regularity, and the present diffusion problem has $F\equiv0$.  Moreover
$X_{\gamma,p}\hookrightarrow X_{\gamma,\mu}$ compactly.

Fix $0<t_*<T_{\max}$.  By continuity of the strong solution on compact preterminal intervals and
\eqref{eq:HMPW-terminal-assumption},
\begin{equation}\label{eq:HMPW-full-tail-Xgamma}
 \sup_{t_*\le t<T_{\max}}\|z(t)\|_{X_{\gamma,p}}<\infty.
\end{equation}
The state constraint supplies the second hypothesis in the abstract continuation theorem uniformly.
Indeed, compactness of $\mathcal D$ and relative openness of $\mathcal V$ in the affine plane
$\mathbf e/3+\mathbb E$ give
\[
 \delta_{\rm st}:=
 \operatorname{dist}_{\mathbb R^3}
 \bigl(\mathcal D,(\mathbf e/3+\mathbb E)\setminus\mathcal V\bigr)>0.
\]
Since $\mu>\mu_0$ implies $X_{\gamma,\mu}\hookrightarrow C(\overline\Omega;\mathbb E)$, let
$C_{\rm emb}$ be the embedding constant.  If $w\in X_{\gamma,\mu}$ satisfies
\[
 \|w-z(t)\|_{X_{\gamma,\mu}}<\delta_{\rm st}/C_{\rm emb},
\]
then $\|w-z(t)\|_{C}<\delta_{\rm st}$ and therefore
$w(x)+\mathbf e/3\in\mathcal V$ for every $x\in\overline\Omega$; hence $w\in V_\mu$.
Thus the whole $X_{\gamma,\mu}$--ball of that radius about $z(t)$ lies in $V_\mu$, and
\begin{equation}\label{eq:HMPW-state-distance}
 \inf_{t_*\le t<T_{\max}}
 \operatorname{dist}_{X_{\gamma,\mu}}\bigl(z(t),\partial V_\mu\bigr)
 \ge \delta_{\rm st}/C_{\rm emb}>0.
\end{equation}

Let $\widetilde z(s):=z(t_*+s)$ and let $t_\mu^+$ denote the maximal existence time of the weighted
solution restarted from $z(t_*)$ at $s=0$.  Uniqueness identifies it with $\widetilde z$ as long as
the original strong solution exists, so $t_\mu^+\ge T_{\max}-t_*$.  Under the standing assumption
$T_{\max}<\infty$, equality holds:
\[
 t_\mu^+=T_{\max}-t_*.
\]
Indeed, if $t_\mu^+>T_{\max}-t_*$, choose a compact interval crossing
$s=T_{\max}-t_*$ and staying a positive distance from the restarted initial time.  The proof of
Theorem~A.1(d) in \cite{HerbergMeyriesPruessWilke2017} explicitly uses the embedding
$\mathbb E_{1,\mu}(0,T)\hookrightarrow\mathbb E_{1,1}(\delta,T)$ for every $\delta>0$;
the theorem additionally records the resulting positive-time $C^1(X_{\gamma,p})$ and fractional
$X_1$ regularity.  Hence on the chosen crossing interval the weighted solution belongs to the
unweighted strong class $W^{1,p}(X_0)\cap L^p(X_1)$.  By uniqueness this is the original strong
solution before the crossing and an extension of it afterward, contradicting maximality at
$T_{\max}$.

With this equality established, every time $0\le s<t_\mu^+$ corresponds to a physical time
$t_*\le t<T_{\max}$.  Therefore \eqref{eq:HMPW-state-distance} gives the all-time distance hypothesis of HMPW Theorem~A.3 in the restarted variable:
\[
 \operatorname{dist}_{X_{\gamma,\mu}}
 \bigl(\widetilde z(s),\partial V_\mu\bigr)
 \ge \delta_{\rm st}/C_{\rm emb}>0,
 \qquad 0\le s<t_\mu^+.
\]
Also, for any fixed $\sigma\in(0,t_\mu^+)$,
\eqref{eq:HMPW-full-tail-Xgamma} and strong-solution continuity give
\[
 \widetilde z\in BC([\sigma,t_\mu^+);V_\mu\cap X_{\gamma,p}).
\]
Thus the two orbit hypotheses in Theorem~A.3 are satisfied.  The compact embedding
$X_{\gamma,p}\hookrightarrow X_{\gamma,\mu}$ and the $C^1$/maximal-regularity properties from
HMPW Section~3.1 supply its remaining assumptions.  The theorem gives $t_\mu^+=\infty$, contradicting
$t_\mu^+=T_{\max}-t_*<\infty$.  Thus $T_{\max}=\infty$.
\end{proof}

\begin{remark}[Logical dependence of the estimates]\label{rem:no-hidden-bootstrap}
The estimates are used in the following order.  The HMPW local theory gives a maximal-$L^p$
strong solution, preservation of the closed simplex, and instantaneous positivity of every
nontrivial species.  On a positive-time separated root chart, the Sobolev/Bochner chain rule gives
the strong root equations.  The invariant rectangle, all one-sided entropy identities, and every
De Giorgi test are then justified directly in this strong class; no classical Schauder estimate is
used earlier in the argument.  Boundary De Giorgi theory gives the first terminal-uniform H\"older modulus.
Only then is the mixing-entropy Hessian frozen on small state-diameter cylinders and the local
Dong--Jeon conormal theorem invoked, yielding a uniform $C^{1+a}$ estimate and hence a spatial
Lipschitz bound.

For terminal trace control, the $C^{1+a}$ bound is used only to obtain compactness of the frozen
principal family
$-C(y')^{-1}\Delta$.  The temporal perturbation of that family is controlled by the earlier
state H\"older modulus, while the spatial Lipschitz bound controls the quadratic forcing generated
by expanding the divergence operator.  Uniform frozen maximal regularity and a fixed-reference-
interval trace estimate then give the affine short-interval recursion in
\eqref{eq:terminal-Xgamma-recursion}.  Finally the HMPW continuation criterion yields the
contradiction.  Neither terminal $C^{2+a}$ regularity, time H\"older continuity of $\nabla y$, nor
compactness of the orbit in the unweighted trace space is used earlier in the argument.
\end{remark}

\begin{lemma}[Zero conserved mass reduces to binary heat flow]\label{lem:zero-mass-binary}
Under the hypotheses of Theorem~\ref{thm:main}, if one conserved species mass vanishes, then either
the state is a constant pure state or the problem reduces globally to a scalar Neumann heat equation.
More precisely, if $\int_\Omega y_3^0\,\dd x=0$ and at least one of the other masses is positive, then
$y_3\equiv0$ and
\[
 (y_1)_t-f_{12}^{-1}\Delta y_1=0,\qquad
 \partial_\nu y_1=0,\qquad y_2=1-y_1.
\]
The analogous statement holds after relabeling.
\end{lemma}

\begin{proof}
Nonnegativity and conservation of mass give $y_3(t)=0$ almost everywhere for every
$t<T_{\max}$; positive-time classicality makes this pointwise for $t>0$.  On such a positive-time
slice the third Maxwell--Stefan relation reads
\[
 0=\nabla y_3
 =-\bigl(f_{31}y_1+f_{32}y_2\bigr)J_3.
\]
Since $y_1+y_2=1$, the coefficient is bounded below by
$\min\{f_{31},f_{32}\}>0$.  Hence $J_3=0$ and $J_1+J_2=0$.  The first relation therefore becomes
\[
 \nabla y_1=f_{12}(y_1J_2-y_2J_1)=-f_{12}J_1,
\]
so the continuity equation is precisely the scalar Neumann heat equation displayed above.
Let $\bar y_1$ be its global strong solution with initial datum $y_1^0$, and put
$\bar y_2=1-\bar y_1$, $\bar y_3=0$.  The maximum principle gives
$0\le\bar y_1\le1$, and the fluxes
\[
 \bar J_1=-f_{12}^{-1}\nabla\bar y_1,
 \qquad \bar J_2=-\bar J_1,
 \qquad \bar J_3=0
\]
satisfy all three Maxwell--Stefan relations.  Local uniqueness therefore identifies this global
binary solution with the original strong solution on their common lifespan.  Hence the original
strong solution cannot have a finite maximal existence time.  If two conserved masses vanish, the remaining component is identically one and the
state is constant.
\end{proof}

\subsection{Pairwise distinct frictions}

\begin{theorem}[Distinct-friction case]
\label{thm:global-all-distinct-d}
Under the hypotheses of Theorem~\ref{thm:main}, assume in addition that the three binary frictions
are pairwise distinct.  Then the maximal strong solution is global.
\end{theorem}

\begin{proof}
If a conserved species mass vanishes, Lemma~\ref{lem:zero-mass-binary} gives global existence.
It remains to consider the case in which all three conserved masses are positive.  Suppose for
contradiction that $T_{\max}<\infty$.  Theorem~\ref{thm:spectral-DG-holder} and
Proposition~\ref{prop:first-Schauder-gain} provide the terminal H\"older and first-derivative
bounds needed in Proposition~\ref{prop:terminal-trace-bound}.  Hence, for some late time $\tau$,
\[
 \sup_{\tau\le t<T_{\max}}
 \|y(t)-\mathbf e/3\|_{X_{\gamma,p}}<\infty.
\]
Proposition~\ref{prop:HMPW-trace-continuation} then gives $T_{\max}=\infty$, contradicting the
assumption.
\end{proof}

\subsection{Coincident frictions}

\begin{lemma}[Two equal frictions give a triangular scalar system]\label{lem:equal-frictions}
Suppose, after relabeling,
\[
f_{12}=f_{13}=a>0,
\qquad f_{23}=c>0.
\]
Then
\begin{equation}\label{eq:equal-heat}
(y_1)_t-a^{-1}\Delta y_1=0,\qquad \partial_\nu y_1=0.
\end{equation}
If
\[
D(y_1)=c+(a-c)y_1,
\]
then the second independent component satisfies
\begin{equation}\label{eq:equal-triangular}
(y_2)_t-\Div\left(
\frac1{D(y_1)}\nabla y_2
+\frac{a-c}{aD(y_1)}y_2\nabla y_1
\right)=0.
\end{equation}
Moreover
\[
\min\{a,c\}\le D(y_1)\le\max\{a,c\}.
\]
Hence the coincident-friction problem is globally strongly solvable for arbitrary compatible
strong data in every finite spatial dimension.
\end{lemma}

\begin{proof}
Since $f_{12}=f_{13}=a$,
\[
\nabla y_1
=a(y_1J_2-y_2J_1)+a(y_1J_3-y_3J_1)
=-aJ_1,
\]
using $J_1+J_2+J_3=0$ and $y_1+y_2+y_3=1$.  Thus $J_1=-a^{-1}\nabla y_1$ and
\eqref{eq:equal-heat} follows.  In the reduced variables $(y_1,y_2)$ the friction matrix and its
inverse are explicitly
\[
C(y')=
\begin{pmatrix}
a&0\\
(c-a)y_2&D(y_1)
\end{pmatrix},
\qquad
C(y')^{-1}=
\begin{pmatrix}
a^{-1}&0\\
\dfrac{(a-c)y_2}{aD(y_1)}&D(y_1)^{-1}
\end{pmatrix}.
\]
Since $(J_1,J_2)^{\sf T}=-C(y')^{-1}\nabla(y_1,y_2)^{\sf T}$, the second continuity equation is exactly
\eqref{eq:equal-triangular}.  Its natural no-flux condition is
\[
\left(\frac1{D(y_1)}\nabla y_2
+\frac{a-c}{aD(y_1)}y_2\nabla y_1\right)\!\cdot\nu=0.
\]
Since \eqref{eq:equal-heat} gives $\partial_\nu y_1=0$ and $D(y_1)>0$, this is exactly
$\partial_\nu y_2=0$.  The coefficient $D(y_1)$ stays between $a$ and $c$ because
$0\le y_1\le1$.  The first equation is the Neumann heat equation; its global strong solution is
obtained on every finite interval from the constant-coefficient case of
\cite[Theorem~6.2 in the preprint version, with $q=p$]{Augner2024}.  Let $y_1$ denote that solution.
Uniqueness of the scalar heat problem identifies it with the first component of the original
Maxwell--Stefan solution on their common preterminal interval.  By the maximum principle,
$0\le y_1\le1$, and hence
\[
\min\{a,c\}\le D(y_1)\le\max\{a,c\}
\]
for all time.  Thus the coefficients in the second equation are determined by the global heat solution, and the
equation for $y_2$ is linear.

We now exclude a finite maximal existence time directly, without a terminal compactness argument.  Suppose
the original Maxwell--Stefan solution had a finite maximal time $T_{\max}$, and fix
$0<\tau<T_{\max}$.  The original strong solution supplies the trace
\[
y_2(\tau)\in W_p^{2-2/p}(\Omega),\qquad \partial_\nu y_2(\tau)=0.
\]
For every prescribed finite $T>\tau$, the global heat solution satisfies
\[
y_1\in C^1([\tau,T];C(\overline\Omega))\cap C([\tau,T];C^2(\overline\Omega)).
\]
This positive-time regularization is also the constant-friction heat case of
\cite[Theorem~3.3]{HerbergMeyriesPruessWilke2017}: apply that theorem to the auxiliary heat
state $(y_1,1-y_1,0)$ with all three frictions equal to $a$.
Set
\[
A_0=D(y_1)^{-1},
\qquad
B_0=\frac{a-c}{aD(y_1)}\nabla y_1.
\]
Then the triangular equation is equivalently
\begin{equation}\label{eq:equal-linear-nondiv}
(y_2)_t-A_0\Delta y_2
-(\nabla A_0+B_0)\cdot\nabla y_2
-(\Div B_0)y_2=0,
\qquad \partial_\nu y_2=0.
\end{equation}
On $[\tau,T]\times\overline\Omega$, the principal coefficient is continuous and uniformly
elliptic, while both lower-order coefficients are continuous and bounded: $\nabla A_0+B_0$
contains only first derivatives of $y_1$, and $\Div B_0$ contains at most second derivatives
and quadratic products of first derivatives.  The finite-interval theorem
\cite[Theorem~6.2 in the preprint version, with $q=p$]{Augner2024} therefore applies with
$E=\mathbb C$, $m=m_1=1$, and homogeneous Neumann data.  More explicitly, its lower-order
coefficient condition can be taken with both integrability indices equal to $p$, since
\[
 \frac1p+\frac d{2p}<\frac12\qquad (p>d+2).
\]
The boundary coefficient and initial compatibility conditions are the scalar versions of those
checked in Proposition~\ref{prop:terminal-trace-bound}.  The scalar principal symbol is positive,
and the Neumann complementing condition is the one-component case of the same boundary-symbol
calculation.  No sign restriction on $\Div B_0$ is required.  We obtain a unique real solution
\[
\begin{split}
\widetilde y_2\in{}&W_p^1(\tau,T;L_p(\Omega))\cap L_p(\tau,T;W_p^2(\Omega))\\
 &{}\cap C([\tau,T];W_p^{2-2/p}(\Omega)),\qquad \partial_\nu\widetilde y_2=0.
\end{split}
\]
Reality follows by uniqueness and complex conjugation.  On every compact subinterval of
$[\tau,T_{\max})$, the original component solves the same linear initial--Neumann problem;
uniqueness gives $\widetilde y_2=y_2$.  Taking $T>T_{\max}$ thus provides a strong continuation
through the assumed terminal time.

It remains to verify that this continuation stays in the physical concentration simplex.  Put
$S=1-y_1$.  Since
\[
-\frac1{D(y_1)}+\frac{a-c}{aD(y_1)}(1-y_1)=-\frac1a,
\]
the heat equation for $y_1$ gives the exact identity
\[
S_t-\Div(A_0\nabla S+B_0S)=0.
\]
Consequently $\widetilde y_3:=S-\widetilde y_2$ satisfies the same homogeneous scalar linear
problem as $\widetilde y_2$.  Positivity follows directly from that divergence equation, without another regularity theorem.  If $w$ denotes either $\widetilde y_2$ or $\widetilde y_3$, then
\[
 w_t-\Div(A_0\nabla w+B_0w)=0,
 \qquad
 (A_0\nabla w+B_0w)\cdot\nu=0,
\]
because $B_0\cdot\nu=0$.  Testing the weak equation with $-w_-$ gives
\[
 \frac12\frac{\dd}{\dd t}\|w_-(t)\|_{L^2}^2
 +\int_\Omega A_0|\nabla w_-|^2
 =-\int_\Omega B_0w_-\cdot\nabla w_-.
\]
Uniform ellipticity of $A_0$ and Young's inequality imply
\[
 \frac{\dd}{\dd t}\|w_-(t)\|_{L^2}^2
 \le C_T\|w_-(t)\|_{L^2}^2
\]
on every finite interval $[\tau,T]$.  Since both initial negative parts vanish at $t=\tau$,
Gronwall's lemma yields $\widetilde y_2,\widetilde y_3\ge0$.  Therefore
\[
\widetilde y_2\ge0,
\qquad
\widetilde y_3\ge0,
\qquad
 y_1+\widetilde y_2+\widetilde y_3=1
\]
on every finite interval.  Moreover $D(y_1)$ never vanishes, so the displayed triangular inverse of
$C(y')$ remains valid and reconstructs the full Maxwell--Stefan fluxes.  Thus the continuation is a
strong solution of the original problem, contradicting maximality.  Hence the triangular problem
is global, and its positive-time classicality follows from
\cite[Theorem~3.3]{HerbergMeyriesPruessWilke2017}.  If $a=c$, all three species satisfy heat
equations with diffusivity $1/a$.
\end{proof}

\begin{proof}[Proof of Theorem~\ref{thm:main}]
If a conserved species mass vanishes, Lemma~\ref{lem:zero-mass-binary} applies, independently of
any relation among the friction coefficients.  Hence assume all three conserved masses are positive.
If the three frictions are pairwise distinct, Theorem~\ref{thm:global-all-distinct-d} applies.  If
exactly two coincide, relabel the species and apply Lemma~\ref{lem:equal-frictions}.  If all three
coincide, each component solves a Neumann heat equation.  This exhausts all positive friction triples
and all nonnegative compatible data.  Uniqueness follows by patching the local uniqueness intervals
supplied by the maximal-regularity theory.
\end{proof}

\section{Concluding remarks}
The proof uses ternary-specific spectral coordinates: the trace and determinant of the $2\times2$
reduced friction matrix produce two scalar roots and the Euler--Poisson--Darboux identity used to
construct the truncation entropies.  The theorem also allows initial data on the boundary of the concentration simplex: solutions with
positive conserved species masses are expressed in spectral coordinates only at positive times,
whereas a species with zero conserved mass reduces the system to a binary heat flow.

For four or more species, arbitrary constant friction coefficients need not admit a pair-sum
representation.  The present proof therefore does not extend directly to the corresponding global
problem.  Nonideal thermodynamic driving forces provide a separate extension.  For these models, the
principal diffusion matrix also depends on the thermodynamic constitutive law, even when the
friction coefficients are constant.  The spectral-coordinate reduction and Euler--Poisson--Darboux
entropy identities used here therefore require separate verification.
The proof uses the methods developed in the additive-friction preprint
~\cite{BotheAdditive2026}, which establishes the moment and root coordinates, multi-EPD identities,
entropy stabilization, and De Giorgi estimates for general $N$ under the positive additive
restriction $f_{ij}=g_i+g_j$.
The present paper extends those methods to the unrestricted ternary problem rather than deriving
a separate regularity argument.  It proves that, for $N=3$, the spectral geometry remains valid
for every positive physical friction triple, including those whose unique real pair-sum generators
are not all positive, and treats closed-simplex data and coincident frictions.  The two spectral descriptions are affinely
identical, as shown in Section~\ref{sec:spectral}; consequently the present theorem removes the
positive-generator restriction in the ternary problem and covers every positive ternary friction
triple.

The Darken representation also describes the coefficient class covered by the theorem.  By
\eqref{eq:ternary-Darken-tracers}, every constant ternary friction triple belongs to the fixed-ray
core-diagonal class $d(y)=\Pi(y)^{-1}(\gamma,\beta,\alpha)^{\sf T}$, and the roots $u,v$ are exactly
the Stieltjes roots for the fixed poles $(\gamma,\beta,\alpha)$; the tracer-mobility
roots are $u/\Pi=1/v$ and $v/\Pi=1/u$.  For comparison, if the frictions depend on composition,
\[
 a(y)=f_{12}(y),\qquad b(y)=f_{13}(y),\qquad c(y)=f_{23}(y)>0,
\]
then the same pointwise core-diagonal inversion gives
\[
 \Delta(y)=y_1a(y)b(y)+y_2a(y)c(y)+y_3b(y)c(y),\qquad
 d(y)=\frac1{\Delta(y)}(c(y),b(y),a(y))^{\sf T}.
\]
Thus general positive composition-dependent ternary frictions still define a Darken system, but the
tracer vector need not lie on a fixed ray, and the Stieltjes poles may depend on composition.
Such systems are not covered by the present constant-friction theorem or its fixed-pole argument;
extending the result to them requires further analysis.

The theorem concerns the normalized ideal isothermal, isobaric Maxwell--Stefan diffusion model with equal molar masses,
unit total density, zero reaction, and constant binary frictions.  Reactive, nonideal, nonisothermal,
compressible, unequal-molar-mass, and variable-friction systems are not treated here.

\section*{Disclosure of computational assistance}
OpenAI's ChatGPT was used as an auxiliary tool for language editing, algebraic and proof
cross-checks, manuscript revision, and bibliographic checks.  All mathematical statements, arguments, and references remain the
responsibility of the author.

\end{document}